\documentclass{article}[12pt]
\usepackage[english]{babel}
\usepackage[utf8]{inputenc}
\usepackage[english]{babel}
\usepackage{amsmath}
\usepackage[top=2cm, bottom=2cm, left=2cm, right=2cm]{geometry}

\usepackage{amsthm}
\usepackage{amsfonts}
\usepackage{array}
\usepackage{amssymb}
\usepackage{mathtools}
\usepackage{enumerate}
\usepackage{enumitem}
\usepackage{comment}
\usepackage{bbm}
\usepackage{tikz}
\usepackage{pgfplots}
\usepackage{graphicx}
\usepackage{comment}
\usepackage{xcolor}
\usepackage{float}
\usepackage{verbatim}
\usepackage{eurosym}
\usepackage[noend]{algpseudocode}
\usepackage{algorithm}
\let\oldReturn\Return
\renewcommand{\Return}{\State\oldReturn}
\usetikzlibrary{plotmarks}
\numberwithin{equation}{section}
\DeclareMathOperator*{\argmin}{arg\,min}
\DeclareMathOperator*{\argmax}{arg\,max}

\newtheorem{assumption}{Assumption}
\newtheorem{remark}{Remark}[section]

\numberwithin{equation}{section} 
\newtheorem{proposition}{Proposition}[section]
\newtheorem{lemma}{Lemma}[section]

\newtheorem{theorem}{Theorem}[section]
\newtheorem{corollary}{Corollary}[section]

\providecommand{\keywords}[1]{\textbf{\textit{Keywords:}} #1}
\date{}

\usepackage[skins]{tcolorbox}
\tcolorboxenvironment{lemma}{blanker,before skip=10pt,after skip=10pt}

\def\E{\mathbb{E}}

\def\X{\mathbb{X}}
\def\U{\mathbb{U}}
\everymath{\displaystyle}

\begin{document}

\title{Decomposition of convex high dimensional aggregative 
stochastic control problems\thanks{
The first, second, third and fifth author thank the FiME Lab (Institut Europlace de Finance). The third author was supported by the PGMO project ``Optimal control of conservation equations", itself supported by iCODE(IDEX Paris-Saclay) and the Hadamard Mathematics LabEx.
\\
$^{1}$Adrien Seguret is with PSL Research University, Universite Paris-Dauphine, CEREMADE, Place de Lattre de Tassigny, 75016 Paris, France,
and  with Finance for Energy Market Research Centre (FIME), Paris, France,
 and with Osiris, EDF R\&D, 91120 Palaiseau, France
        {\tt\small adrien.seguret@edf.fr}\\%
$^{2}$ Clémence Alasseur is with
               Osiris, EDF R\&D, 91120 Palaiseau, France and  with Finance for Energy Market Research Centre (FIME)
             {\tt\small clemence.alasseur@edf.fr}\\
$^{3}$J. Frédéric Bonnans is with Disco Team, L2S, CentraleSupelec/Université Paris-Saclay and Inria-Saclay, France
{\tt\small Frederic.Bonnans@inria.fr}\\
$^{4}$Antonio De Paola is with Department of Electrical and Electronic Engineering, Imperial College London, London, UK
 {\tt\small antonio.de-paola09@imperial.ac.uk}\\
 $^{5}$Nadia Oudjane is with Osiris, EDF R\&D, 91120 Palaiseau, France and  with Finance for Energy Market Research Centre (FIME), Paris, France
        {\tt\small nadia.oudjane@edf.fr}\\
$^6$ Vincenzo Trovato is with Department of Civil, Environmental and Mechanical Engineering, University of Trento, Trento, Italy  and with Department of Electrical and Electronic Engineering, Imperial College London, London, UK 
 {\tt\small vincenzo.trovato@unitn.it}\\
\today}
}
\author{Adrien Seguret$^{1}$ \and Clemence Alasseur$^{2}$  \and {J. Frédéric Bonnans}$^{3}$ \and Antonio De Paola$^{4}$ \and Nadia Oudjane$^{5}$ \and Vincenzo Trovato$^{6}$
}

\maketitle

\begin{abstract}
We consider the framework of convex high dimensional stochastic control problems, in which the controls are aggregated in the cost function. As first contribution, we introduce a modified problem, whose optimal control is under some reasonable assumptions an $\varepsilon$-optimal solution of the original problem. As
second contribution, we present a decentralized algorithm whose convergence to the solution of the modified problem is established. Finally, we study the application of the developed tools in an engineering context, studying a coordination problem for large populations of domestic thermostatically controlled loads (TCLs).
\end{abstract}


\section{Introduction}

The present article aims at solving a high dimensional stochastic control problem $(P_{1})$ involving a large number $n$ of agents indexed by $i\in \{1,\cdots, n\}$, of the form: 
\begin{equation}
\begin{array}{r l}
   (P_{1})  &  
   \left\{
\begin{array}{l}
{\displaystyle \min_{u\in\mathcal{U}} J(u)}\\
{\displaystyle 
J(u):=
\E\left[F_0(\frac{1}{n}\sum_{i=1}^n u^i(\omega^i,\omega^{-i}))
+\frac{1}{n}\sum_{i=1}^{n} 
G_{i}(u^i(\cdot,\omega^{-i}), \omega^i)\right]}.
\end{array}
\right.
\end{array}
\label{P1}
\end{equation}

Here the noise 
$\omega := (\omega^1, \ldots , \omega^n)$
belongs to $\Omega := \Pi_{i=1}^n \Omega^i$, where $(\Omega^i, \mathcal{F}^i, \mu^i)$ is a probability space, and 
$(\Omega, \mathcal{F}, \mu)$ is the corresponding product probability space. 
Let $\omega^{-i}:=(\omega^1,\ldots,\omega^{i-1},\omega^{i+1},\ldots,\omega^{n})$ denote an element of the space $\Omega^{-i} := \Pi_{j=1,j\neq i}^n \Omega^j$. The associated product probability space is $(\Omega^{-i},\mathcal{F}^{-i},\mu^{-i})$, where $\mathcal{F}^{-i}:=\otimes_{j=1,j\neq i}^n \mathcal{F}^{j}$ and $\mu^{-i} := \Pi_{j=1,j\neq i}^n \mu^j$. Each decision variable $u^i$ is a random variable 
(i.e. is $\mathcal{F}$-measurable), square summable with value in a Hilbert space $\mathbb{U}$
so that $u := (u^1,\ldots, u^n)$ belongs to $L^2(\Omega, (\mathbb{U})^n )$. 
The function
$\omega^i \mapsto u^i(\omega^i,\omega^{-i})$
is denoted by $u^i(\cdot,\omega^{-i})$
and is a.s. (in $\omega^{-i}$) $\mathcal{F}^i$-measurable and  belongs to $L^2(\Omega^i, \mathbb{U} )$.
Also, $\mathcal{U} := \Pi_{i=1}^n \mathcal {U}_i$ where $\mathcal {U}_i$ is, for $i=1$ to $n$, 
a closed convex subset of $L^2(\Omega, \mathbb{U})$.
In the application to dynamical problems, the constraint $u^i \in \mathcal{U}_i$
includes the constraint of adaptation of $u^i$ to some filtration.
If each $u^i$ is a random variable of $\omega^i$, for $i=1$ to $n$,
we say that $u$ is a decentralized decision variable. 

The cost function is the sum of a \textit{coupling term} $F_0: \mathbb{U} \to \mathbb{R}$, 
function of the \textit{aggregate strategies} $\frac{1}{n} \sum_{i=1}^n u^i$,
and  \textit{local terms} functions of the local decision $u^i$ and local noise $\omega^i$
with $G_i : L^2(\Omega^i, \mathbb{U}) \times  \Omega^i \to \mathbb{R}$. This framework aims at containing stochastic optimal control problems, where the states of the agents are driven by independent noises (see equations \eqref{def_g_sto_cont} and \eqref{evolution SDE} developed in Section \ref{examples_sto_control}).

\subsection{Motivations}

This work is motivated by its
potential applications  to distributed coordination of large populations of small agents, with relevant real-world  implications in different sectors, from communication networks to power systems. 
The application developed in this paper deals with
the coordination of flexible electrical appliances, to support power system operation in a context of increasing penetration of renewables.
Among other appliances, thermostatically controlled loads (e.g. refrigerators, air conditioners etc.) have been investigated in the last few years, for their intrinsic flexibility and potential for network support.
Several papers have already assessed the potential of
demand-side response actions for frequency response services of TCLs \cite{short2007stabilization} and how 
the population recovers from significant perturbations \cite{chertkov2017ensemble}. The coordination of TCLs can be performed in a centralized way, like in \cite{hao2014aggregate}. However, this approach raises concerns with respect to the communication requirements and  customer privacy. A common objective can be reached in a fully distributed approach, like in 
\cite{trovato2016leaky}, where each TCL is able to calculate its own actions (ON/OFF switching) to pursue a common objective. 
This paper is related to the work of De Paola \textit{et al.} \cite{depaola2019mean}, where each agent represents a flexible TCL device. 
In \cite{depaola2019mean}
a distributed solution is presented
for the operation of a population of $n=2\times 10^7$  refrigerators providing frequency support and load shifting. They adopt a game-theory
framework, modelling the TCLs as price-responsive
rational agents that schedule their energy consumption and
allocate their frequency response provision in order to minimize their operational costs. 
The potential practical application of our work also considers a large population of TCLS which, as extension to \cite{depaola2019mean}, have stochastic dynamics. The proposed approach is able to minimize the overall system costs in a distributed way, with each TCL determining its optimal power consumption profile in response to price signals.

\subsection{Related literature}

The considered problem belongs to the class of stochastic control: looking for strategies
minimizing the expectation of an objective function under specific constraints.
One of the main approaches proposed in the literature to tackle this problem is to use random trees: this  consists in replacing the almost sure constraints, induced by non-anticipativity, by a finite number of constraints, in order to get a finite set of scenarios (see \cite{higle2013stochastic} and \cite{ruszczynski2003stochastic}).
Once the tree structure is built, the problem is solved by different decomposition methods such as scenario decomposition \cite{rockafellar1991scenarios}  or dynamic splitting \cite{salinger1997splitting}. The main objective of the scenario method is to reduce the problem to an approximated deterministic one. 
The present paper focuses on high dimensional noise problems with a large number of time steps, for which this approach is not feasible.

The idea of reducing a single high-dimensional problem to a large number of smaller problems with lower dimension has been widely studied in the deterministic case. In deterministic and stochastic problems  there is the possibility of using time decomposition thanks to the Dynamic Programming Principle \cite{bertsekas2004stochastic}, taking advantage of the Markov property of the system.
However, this method requires a specific time structure of the cost function and is not suitable for
problems with high-dimensional state spaces.
Under continuous linear-convex assumptions, one can deal with the curse of dimensionality by using the Stochastic Dual Dynamic Programming algorithm (SDDP) \cite{pereira1991multi} to get upper and lower bounds of the value function, using polyhedral approximations.
The almost-sure convergence of a broad class of SDDP algorithms has been proven \cite{philpott2008convergence}, and complexity of the algorithm can be estimated, in the specific case of Lipschitz continuous value function \cite{lan2020complexity} or by using  a regularization of the value functions \cite{zhang2022stochastic}.
In \cite{leclere2020exact,pacaud2018decentralized}, a stopping  criteria  based  on  a  dual version of SDDP, which gives a deterministic upper-bound for the primal problem, is proposed. SDDP is well-adapted for medium sized population problems ($n\leq 30$), whereas it fails for problems with   large populations ($n>1000$) such as the ones considered in this paper. To tackle this type of high dimensional problems, it is natural to investigate decomposition techniques in the spirit of the Dual Approximation Dynamic Programming (DADP)
\cite{girardeau2010solving,leclere2014contributions}.
This approach is characterized by a price decomposition of the problem, where the stochastic constraints are projected on subspaces such that the associated Lagrangian multiplier is adapted for dynamic programming. Then the optimal multiplier is estimated by implementing Uzawa's algorithm. To this end in  \cite{leclere2014contributions},  the Uzawa's algorithm, formulated in a Hilbert setting, is extended to a Banach space. DADP has been applied in different cases, such as storage management problem for electrical production in  \cite{barty2010decomposition,girardeau2010solving} and hydro valley management \cite{carpentier2018stochastic}. The idea of approaching the primal and dual problems by restricting or relaxing the set of decision variables has also been proposed in the context of stochastic programming \cite{brown2010information,kuhn2011primal} to provide upper and lower bounds for the considered problem. In the proposed paper, in the same vein as DADP, we propose a price  decomposition approach restricted to deterministic prices. This new approach takes advantage of the large population number in order to introduce an auxiliary problem where the coupling term is purely deterministic.

\subsection{Contributions}

The numerical difficulty of Problem $(P_1)$ is related to the randomness of the aggregate term $\textstyle\frac{1}{n}\sum_{i=1}^nu^i$ involved in the coupling function $F_0$.
Let us introduce the set of decentralized controls:
\begin{equation}
    \hat{\mathcal{U}}:=\prod_{i=1}^n \hat{\mathcal{U}_i},  
    \mbox{where }\hat{\mathcal{U}_i}:=\{u^i\in\mathcal{U}_i\,| \,
u^{i}\, \mbox{ is }\mathcal{T}^1\otimes \ldots \otimes \mathcal{F}^i \otimes  \mathcal{T}^{i+1}\otimes \ldots \otimes \mathcal{T}^n-\mbox{measurable}\},
\end{equation}
where $\mathcal{T}^j$ is the trivial $\sigma$-field $\{\emptyset, \Omega^j\}$. 
Note that by construction,
we can identify $\hat{\mathcal{U}_i}$ with $\mathcal{F}^i-$measurable functions defined on $\Omega^i$. In addition, two decentralized controls $u^i\in \hat{\mathcal{U}}_i$ and $u^j\in \hat{\mathcal{U}}_j$, $i\neq j$, are independent random variables.
If the control $u$ is decentralized, since the variance of the sum of independent random variables is equal to the sum of the variances of these random variables,  the aggregate term $\textstyle \frac{1}{n}\sum_{i=1}^nu^i$ can be approximated by  $\textstyle\frac{1}{n}\sum_{i=1}^n\E\,u^i$ when $n$ is large enough.
Let us consider then the following approximation of Problem $(P_1)$:
\begin{equation}
\label{eq:J modified}
\begin{array}{r l}
   (P_{2})  &  
   \left\{
\begin{array}{l}
 \min_{u\in\mathcal{U}} \tilde J(u)
 \\
\tilde J(u):= F_{0}\left(\frac{1}{n}\sum_{i=1}^n \E\, u^i \right) 
+\frac{1}{n}\E\left[\sum_{i=1}^n  
G_{i}(u^i(\cdot,\omega^{-i}), \omega^i)\right].
\end{array}
\right.
\end{array}
\end{equation}
A first step consists  in showing that, without loss of optimality in Problem ($P_2$), one can restrict the control set $\mathcal{U}$ to $\hat{\mathcal{U}}$.

Theorem \ref{independence of ustar} states the equivalence between Problem $(P_2)$ and its decentralized version $(\hat{P}_2)$ defined by:
\begin{equation}
\label{P_2_hat}
\begin{array}{r l}
   (\hat{P}_{2})  &  
   \left\{
\begin{array}{l}
 \min_{u\in\hat{\mathcal{U}}} \tilde J(u)
 \\
\tilde J(u):= F_{0}\left(\frac{1}{n}\sum_{i=1}^n \E\,u^i\right) 
+\frac{1}{n}\E\left[\sum_{i=1}^n  
G_{i}(u^i, \omega^i)\right].
\end{array}
\right.
\end{array}
\end{equation}
Through the article, the circumflex symbol $\hat{}$ will be used to denote minimization problems w.r.t. decentralized controls.
Problem $(\hat{P}_{2})$ 
can be written as:
\begin{equation}\label{eq:P2v}\begin{array}{r l}
   (\hat{P}_{2}')  &  \left\{
    \begin{array}{l}
        \underset{u\in\hat{\mathcal{U}},  v\in \mathbb{U} }{\min}\bar{J}(u,v), \\
    \bar{J}(u,v) :=
    F_{0}(v)+\frac{1}{n}\mathbb{E}\left[\sum_{i=1}^{n}
   G_{i}(u^{i},\omega^i)
    \right],
    \\
    \mbox{ s.t  } \,  g(u,v)=0,\\ 
    \end{array}
    \right.
    \end{array}
    \end{equation}  
where $g: \mathcal{U}\times \mathbb{U} \rightarrow \mathbb{U}$
is defined by 
\begin{equation}
    \label{def_g}
    g(u,v):=\frac{1}{n}\sum_{i=1}^{n}\mathbb{E}\,u^{i} -v.
\end{equation}
Observe that, for any $u^i\in \hat{\mathcal{U}}_i$, $ G_{i}(u^{i},\cdot)$ is independent of $\mathcal{F}^{-i}$.
As a first contribution, this paper shows that under some convexity and regularity assumptions on $F_{0}$ and $(G_{i})_{i\in\{1,\ldots,n\}}$, any solution of Problem $(P_{2})$ is an $\varepsilon_{n}$-solution of $(P_{1})$, with $\varepsilon_{n}\rightarrow 0$ when $n\to \infty$.
In addition, we will see that an approach of price decomposition for $(\hat{P}_2)$, based on the formulation $(\hat{P}_{2}')$, is tractable for dynamical problems, since the problem of minimizing the Lagrangian with deterministic dual variables can be decomposed in subproblems which are solvable by Dynamic Programming.

Since computing the dual cost of $ (\hat{P}_{2}') $ is expensive, we propose  \textit{Stochastic Uzawa} and \textit{Sampled Stochastic Uzawa} algorithms relying on the Robbins-Monroe algorithm, in the spirit of the stochastic gradient. Their convergence is established, relying on the proof provided by \cite{geiersbach2019projected} for the convergence of the stochastic gradient in a Hilbert space.
We check the effectiveness of the \textit{Stochastic Uzawa} algorithm on a linear quadratic Gaussian framework,
and we apply the \textit{Sampled Stochastic Uzawa} algorithm to a model of power system, inspired by the work of A. De Paola \textit{et al.} \cite{depaola2019mean}.

\subsection{Assumptions}

Various assumptions needed in the article are listed in this subsection.
\begin{assumption}
\label{Assum_framework}
\begin{enumerate}[label={\normalfont (\roman*)}]
\item Each set $\mathcal{U}_i$ is bounded, i.e. there exists $M>0$ such that $\E \Vert u^i \Vert_{\mathbb{U}}^2 \leq M^2$, for $i\in \{1,\ldots,n\}$.  

\item \label{convexity:Gi}The function $u^i\mapsto G_i(u^i(\cdot,\omega^{-i}),\omega^i)$ is a.s. non negative, convex and lower semi continuous (l.s.c. for short). And, for any $u^i\in \mathcal{U}_i$, the function $\omega\mapsto G_i(u^i(\cdot,\omega^{-i}),\omega^i)$ is measurable.

\item \label{F_0_lsc} The function $F_0$ is l.s.c. and proper.

\item\label{P1:feasible}
Problem $(P_1)$ is feasible.
\end{enumerate} 
\end{assumption}

\begin{assumption}\label{F_0_conv}
The function $F_0$ is convex.
\end{assumption}

\begin{assumption}
\label{lip der assu}
The function $F_{0}$ is Gâteaux differentiable with $c$-Lipschitz derivative.
\end{assumption}

\begin{assumption}\label{assumption_conv_algo}
\begin{enumerate}[label={\normalfont (\roman*)}]
\item \label{strict_gi}The function $u^i\mapsto G_i(u^i,\omega^i)$ is for a.a. $\omega^i\in\Omega^i$ strictly convex on $\hat{\mathcal{U}}_i$.
\item \label{quad_growth}The function $F_0$ 
has at least quadratic growth, i.e. there exist $C_1,C_2>0$ such that for any $v\in \mathbb{U}$:
\begin{equation*}
    C_1\Vert v \Vert_\mathbb{U}^2-C_2\leq F_0(v).
\end{equation*}
\end{enumerate}
\end{assumption}
\begin{remark}
By Lemma \ref{quad_growth_lemma} in Appendix \ref{append}, if $F_0$ satisfies Assumption \ref{lip der assu} then, $F_0$ has at most quadratic growth, i.e. there exists $C>0$ such that for any $v\in \mathbb{U}$ one has:
\begin{equation*}
     F_0(v)\leq C(\Vert v \Vert_\mathbb{U}^2+1).
\end{equation*}
\end{remark}
We denote by $\{\rho_k\}_{k\in \mathbb{N}^\ast}$ the sequence of step sizes used in the \textit{Stochastic Uzawa} and \textit{Sampled Stochastic Uzawa} algorithms in Section \ref{proposed algo}.
\begin{assumption}
\label{assumption 1}
The sequence $\{\rho_{k}\}_{k\in \mathbb{N}^\ast}$ is such that: $\rho_{k}> 0$, 
$\sum_{k=1}^{\infty}\rho_{k}=\infty$ and $\sum_{k=1}^{\infty} (\rho_{k})^{2}<\infty$.
\end{assumption}
Note that a sequence of the form $\rho_{k}:=\frac{a}{b+k}$, with $(a,b)\in\mathbb{R}^{\ast}_{+}\times\mathbb{R}_{+}$, satisfies Assumption \ref{assumption 1}.

\begin{assumption}
\label{strong conv hyp}
\begin{enumerate}
[label={\normalfont (\roman*)}]
\item \label{strong conv F0}
$F_{0}$ is strongly convex.
\item  \label{strong conv Fi}
There exists $\delta>0$ such that for any $i\in\{1,\ldots,n\}$ and for a.a. $\omega^i\in \Omega^i$, the function $\hat{\mathcal{U}}_i\ni u^{i}\mapsto \,G_{i}(u^{i},\omega^i)$ is strongly convex with modulus of convexity greater or equal to $\delta$.
\end{enumerate}
\end{assumption}
Obviously, Assumption \ref{strong conv hyp} is stronger than Assumption \ref{assumption_conv_algo}.

\section{Approximating the optimization problem}
\label{approx_results}
In this section, the link between the values of problems $(P_1)$ and $(P_2)$ is analyzed.

\begin{lemma}
\label{existence solution P1}
Let Assumptions \ref{Assum_framework} and \ref{F_0_conv} hold.
Then Problem $(P_1)$ has a solution, i.e. $J$ reaches its minimum over $\mathcal{U}$.
\end{lemma}

\begin{proof}
The existence of a minimum is proved by considering a minimizing sequence (which exists since ($P_1$) is feasible) $\{u_{k}\}$ of $J$ over $\mathcal{U}$. The set $\mathcal{U}$ being bounded and weakly close, there exists a subsequence $\{u_{k_{\ell}}\}$ which weakly converges to a certain $u^{\ast}\in \mathcal{U}$. Using Assumptions \ref{Assum_framework}.\ref{convexity:Gi} and convexity of $F_{0}$, it follows that $\lim\inf J(u_{k_{\ell}})\geq J(u^{\ast})$ and thus $u^{\ast}$ is a solution of $(P_{1})$. 
\end{proof}
We obtain the following Corollary about $(P_2)$.
\begin{corollary}\label{corro_value_P_2}
If Assumptions \ref{Assum_framework} and \ref{F_0_conv} are satisfied,
then Problem $(P_2)$ has a solution and its value is lower or equal to the value of Problem $(P_1)$ i.e:
\begin{equation*}
\underset{u\in\mathcal{U}}{\inf}\,\tilde{J}(u)
\leq 
\underset{u\in\mathcal{U}}{\inf}\,{J}(u).
\end{equation*}
\end{corollary}
\begin{proof}
Assumption \ref{Assum_framework}.\ref{P1:feasible} and convexity of $F_0$ imply that ($P_2$) is feasible. By using the same techniques as in the proof of Lemma \ref{existence solution P1}, one can prove that $(P_2)$ admits a solution. Using the convexity of $F_0$ and Jensen's inequality, one has for any centralized control $u\in\mathcal{U}$:
\begin{equation*}
F_{0}(\frac{1}{n}\sum_{i=1}^n \E\,u^i)
\leq \E[F_{0}(\frac{1}{n}\sum_{i=1}^{n} u^i)],
    \label{Jensen ineq used}
\end{equation*}
and the conclusion follows from the definition of $(P_1)$ in \eqref{P1} and $(P_2)$ in \eqref{eq:J modified}
\end{proof}

We have the following key result.
\begin{theorem} 
\label{independence of ustar}
If Assumption \ref{Assum_framework} is satisfied, then the \textit{decentralized} Problem $(\hat{P}_2)$
has the same value as the \textit{centralized} Problem $({P}_2)$
i.e.:
\begin{equation}
\label{min on A}
\underset{u\in\hat{\mathcal{U}}}{\inf}\,\tilde{J}(u)
=
\underset{u\in\mathcal{U}}{\inf}\,\tilde{J}(u).
\end{equation}
\end{theorem}

\begin{proof}
Since $\hat{\mathcal{U}}\subset\mathcal{U}$, it is immediate that $\underset{u\in\mathcal{U}}{\inf}\,\tilde{J}(u)\leq \underset{u\in\hat{\mathcal{U}}}{\inf}\,\tilde{J}(u)$.

Fix $i\in\{1,\ldots,n\}$, using the definition of conditional expectation, we define $\tilde{u}^i\in L^2(\Omega^i, \mathbb{U})$ for any $u^i\in \mathcal{U}_i$ by:
\begin{equation*}
\tilde{u}^i(\omega^i):=\mathbb{E}[u^i(\omega^i, \omega^{-i})\vert \omega^i]=\int_{\Omega^{-i}} u^i(\omega^i, \omega^{-i})\mathrm{d}\mu^{-i}(\omega^{-i})\quad \mbox{ for any }\omega_i \in \Omega^i.
\end{equation*}
Since $G_i$ is a.s. convex w.r.t. the first variable, Jensen's inequality gives:
\begin{equation}
\label{ineq for indep}
G_i(\tilde{u}^i,\omega^i)\leq 
\int_{\Omega^{-i}} G_i(u^i(\cdot, \omega^{-i}),\omega^i)\mathrm{d}\mu^{-i}(\omega^{-i})=\E[G_i(u^i(\cdot,\omega^{-i}),\omega^i)\vert \omega^i]\quad\mbox{a.s.}
\end{equation}
On the other hand \sloppy $(u^{1},\ldots,u^{n})\mapsto F_{0}(\frac{1}{n}\sum_{i=1}^{n} \mathbb{E}\,u^{i})$ is invariant when taking the conditional expectation, thus:
\begin{equation*}
F_{0}\left(\frac{1}{n}
\sum_{i=1}^{n}\mathbb{E}\,u^{i}\right)
=F_{0}\left(\frac{1}{n}
\sum_{i=1}^{n}\mathbb{E}\,\tilde{u}^i \right).
\end{equation*}
Taking the expectation of \eqref{ineq for indep}, we have $\underset{u\in\hat{\mathcal{U}}}{\inf}\,\tilde{J}(u)\leq\underset{u\in\mathcal{U}}{\inf}\,\tilde{J}(u)$, and the conclusion follows.
\end{proof}

\begin{remark}
In the applications to stochastic control problems (in discrete and continuous time) 
we have the constraint of having progressively measurable control policies. 
Since the set of progressively measurable policies is 
closed and convex, this enters in the above framework.
In particular, the decentralized policy $\tilde{u}^i$
constructed in the above proof is progressively measurable.
\end{remark}

\begin{remark}
By Theorem \ref{independence of ustar}, for any $\varepsilon>0$ there exists an $\varepsilon$-optimal solution of Problem ($P_2$) that is a decentralized control.
\end{remark}

Before stating the next result, we need to introduce Problem $(\hat{P}_1)$, that corresponds to the minimization of $J$, defined in \eqref{P1}, over the set of decentralized controls $\hat{\mathcal{U}}$:
\begin{equation*}
  (\hat{P}_{1}) \left\{
\begin{array}{l}
\displaystyle \min_{u\in\hat{\mathcal{U}}}\, J(u).
\end{array}
\right.
\end{equation*}

\begin{proposition}
\label{optimality th P1 loc}
If Assumption \ref{Assum_framework} is satisfied and $F_{0}$ is Lipschitz continuous with constant $\gamma$, then any solution of $(\hat{P}_2)$ is an $\varepsilon$-optimal solution of $(\hat{P}_1)$ and, conversely, any solution of $(\hat{P}_1)$ is an $\varepsilon$-optimal solution of $(\hat{P}_2)$, with $\varepsilon=2\gamma M/\sqrt{n}$.
\end{proposition}
\begin{proof}
Since $F_0$ is Lipschitz continuous with Lipschitz constant $\gamma$, it holds for any $x,y\in \mathbb{U}$: $\vert F_{0}(x)-F_{0}(y) \vert \leq \gamma\Vert x-y\Vert_{ \mathbb{U}}$. 
We set for any $u\in\mathcal{U}$:
\begin{equation}
\label{def u hat}
\textbf{u}^{i}:=u^{i}-\mathbb{E}\,u^{i}.
\end{equation}
Using the Lipchitz continuity of $F_0$, one has for any $ u\in\hat{\mathcal{U}}$:
\begin{equation*}
\begin{array}{cl}
     \left\vert  \E\Big[F_{0}\big(\frac{1}{n}\sum_{i=1}^{n} u^i\big)
  -F_{0}\big(\frac{1}{n}\sum_{i=1}^n \E\,u^i\big)\Big]  \right\vert
    & \leq
  \E \left\vert F_{0}\big(\frac{1}{n}\sum_{i=1}^{n} u^i\big)
  -F_{0}\big(\frac{1}{n}\sum_{i=1}^n \E\,u^i \big) \right\vert \\
  & \leq \frac{\gamma}{n}\mathbb{E}\,\Vert\sum_{i=1}^n \textbf{u}^i\Vert_{ \mathbb{U}}.
 \end{array}
 \end{equation*}
Using the Jensen's inequality, for any  $u\in\hat{\mathcal{U}}$, the mutual independence of the centered variables $\textbf{u}_{i}$ and $\textbf{u}_{j}$ for any $j\neq i$ and $\E\Vert \textbf{u}_{i}\Vert_{\mathbb{U}}^2\leq M^2$, we get:
 \begin{equation}
   \frac{\gamma}{n}\mathbb{E}\,\Vert\sum_{i=1}^n \textbf{u}^i\Vert_{ \mathbb{U}} \leq
  \frac{\gamma}{n}\mathbb{E}[\Vert\sum_{i=1}^n \textbf{u}^i
 \Vert^{2}_{\mathbb{U}}]^{\frac{1}{2}} 
 \leq\frac{\gamma}{n^{\frac{1}{2}}}M.
\label{ineq expectation 2}
 \end{equation}

 Let $\hat{u}$ denote a minimizer of $(\hat{P}_2)$, then using \eqref{ineq expectation 2} for the first and last inequality, for any $u\in\hat{\mathcal{U}}$ it holds:
 \begin{equation}
 J(\hat{u})\leq \tilde{J}(\hat{u}) + \frac{\gamma}{n^{\frac{1}{2}}}M
 \leq \tilde{J}(u) + \frac{\gamma}{n^{\frac{1}{2}}}M
  \leq J(u) + \frac{2\gamma}{n^{\frac{1}{2}}}M.
 \end{equation}
 Similarly, if ${u}^\ast$ is a solution of $(\hat{P}_1)$, then for any $u\in\hat{\mathcal{U}}$ one has:
  \begin{equation}
 \tilde{J}({u}^\ast)\leq {J}({u}^\ast) + \frac{\gamma}{n^{\frac{1}{2}}}M
 \leq {J}(u) + \frac{\gamma}{n^{\frac{1}{2}}}M
  \leq \tilde{J}(u) + \frac{2\gamma}{n^{\frac{1}{2}}}M.
 \end{equation}
\end{proof}
\begin{theorem}\label{optimality th} 
Let Assumptions \ref{Assum_framework}, \ref{F_0_conv} and \ref{lip der assu} be satisfied.
Then any solution of Problem $(\hat{P}_{2})$
is an $\varepsilon$-optimal solution $($where $\varepsilon=cM^{2}/n)$ of Problem $(P_{1})$.
\end{theorem}
\begin{proof}
From Corollary \ref{corro_value_P_2} and Theorem \ref{independence of ustar}, one has for any $\hat{u}\in \hat{\mathcal{U}}$ solution of $(\hat{P}_2)$ that:
\begin{equation}\label{control_hat_u_J}
    \tilde{J}(\hat{u})\leq \inf_{u\in \mathcal{U}}\,J(u).
\end{equation}
Since $F_{0}$ is convex, differentiable, with a $c$-Lipschitz derivative, one can derive a.s.:
\begin{equation}
\begin{array}{l  }
 \label{ineq for F0 conv and diff}
      F_0(\frac{1}{n}\sum_{i=1}^n \hat{u}^i) - F_{0}(\frac{1}{n}\sum_{i=1}^n \E\,\hat{u}^i) 
      \\\leq \frac{1}{n} \langle \nabla F_{0}(\frac{1}{n}\sum_{i=1}^n \hat{u}^i) \,
      , \sum_{i=1}^n \textbf{u}^{i}\,\rangle_{\mathbb{U}} 
      
      \\  = \frac{1}{n}\langle \nabla F_{0}(\frac{1}{n}\sum_{i=1}^n \hat{u}^i)
       -\nabla F_{0}(\frac{1}{n}\sum_{i=1}^n \E\,\hat{u}^i) 
       \,, \sum_{i=1}^n \textbf{u}^{i}\,\rangle_{\mathbb{U}}
       + \frac{1}{n}\langle \nabla F_{0}(\frac{1}{n}\sum_{i=1}^n \E\,\hat{u}^i) 
      \,, \sum_{i=1}^n \textbf{u}^{i}\,\rangle_{\mathbb{U}}\\
       \leq \frac{c}{n^{2}}\, \Vert \sum_{i=1}^n \textbf{u}^{i}\, \Vert^{2}_{ \mathbb{U}}
     +\frac{1}{n} \langle \nabla F_{0}(\frac{1}{n}\sum_{i=1}^n \E \,\hat{u}^i ) 
      \,, \sum_{i=1}^n \textbf{u}^{i}\,\rangle_{\mathbb{U}},\end{array}
\end{equation}
where $\textbf{u}^i$ is defined from $\hat{u}$ as in \eqref{def u hat}.
From the definition of $\textbf{u}^i$, one obtains:
$$\mathbb{E}\left[\langle
  \nabla F_{0}(\frac{1}{n}\sum_{i=1}^n \E\,\hat{u}^i) 
  \,, \sum_{i=1}^n \textbf{u}^{i}
  \,\rangle_{ \mathbb{U}}\right]  
    = 0.$$
Since $\hat{u}\in \hat{\mathcal{U}}$, controls are mutually independent and bounded a.s. by $M$, one gets as in \eqref{ineq expectation 2}:
\begin{equation}
\label{bound var}
    \frac{c}{n^{2}}\,
    \mathbb{E}
    \Vert \sum_{i=1}^n \textbf{u}^{i}\, \Vert^{2}_{\mathbb{U}} 
   \leq \frac{c}{n}M^{2}.
  \end{equation}  
Taking the expectation of the first and last terms of \eqref{ineq for F0 conv and diff} and then incorporating \eqref{bound var} and the equality above, one obtains:
\begin{equation*}
    J(\hat{u})-\tilde{J}(\hat{u})\leq \frac{c}{n}M^{2}.
\end{equation*}
From previous inequality and \eqref{control_hat_u_J}, we get:
\begin{equation}\label{ineq th epsilon}
    J(\hat{u})\leq \inf_{u\in \mathcal{U}}\,J(u) +\frac{c}{n}M^{2}.
\end{equation}

\end{proof}
\begin{remark}
Observe that the value of the centralized Problem ($P_1$) on the l.h.s. of the inequality \eqref{up_and_lower_bound} below is upper bounded by the following decentralized problem on the r.h.s of this inequality i.e.
\begin{equation}\label{up_and_lower_bound}
\inf_{u\in \mathcal{U}}J(u)\leq\inf_{u\in \hat{\mathcal{U}}}J(u).
\end{equation}
 Ref. \cite{carpentier2020mixed} obtains an upper bound for the decentralized problem and a lower bound for the centralized problem.
The upper bound is provided by a resource decomposition approach (with deterministic quantities) while the lower bound is obtained
by a price decomposition approach with deterministic prices (see Equation (28) of \cite{carpentier2020mixed}).
Theorem \ref{optimality th} provides
an upper bound for Problem ($P_1$) with an a priori quantification of the deviation from the optimal value which vanishes when the number of agents grows to infinity. Moreover, in Section \ref{proposed algo} we provide an original algorithm that allows to approach the solution of the decentralized problem.
\end{remark}

\begin{remark}

Let $\hat{u}$ and $u^{\ast}$ be respectively the optimal solutions of problems $(\hat{P}_{2})$ and $(P_{1})$.
From Jensen's inequality and by definition of $\hat{u}$ we have:
\begin{equation*}
-J(u^{\ast}) \leq -\tilde{J}(u^{\ast})\leq -\tilde{J}(\hat{u}).
\end{equation*}
Adding $J(\hat{u})$, one has:
\begin{equation}
\label{ineq: interest}
    0  \leq J(\hat{u})-{J}(u^{\ast}) \leq 
   J(\hat{u}) -\tilde{J}(u^{\ast}) \leq 
    J(\hat{u})-\tilde{J}(\hat{u}).
\end{equation}
Inequality \eqref{ineq: interest} allows to compute an upper bound of the "optimality" error $J(\hat{u})-J(u^{\ast})$, by evaluating $J(\hat{u})-\tilde{J}(\hat{u})$.
\end{remark}

\section{Dualization and Decentralization of Problem $(P_2)$}
The \textit{Lagrangian} function associated to the constrained optimization Problem $(\hat{P}_{2}')$, defined in \eqref{eq:P2v}, is: $L:\hat{\mathcal{U}}\times  \mathbb{U} \times  \mathbb{U} \xrightarrow{}\bar{\mathbb{R}}$ defined by: 
\begin{equation}\label{def_lag}
 L(u,v,\lambda):= \bar{J}(u,v) + \langle \lambda,\frac{1}{n}\sum_{i=1}^{n}\mathbb{E}\,u^{i} -v \rangle_{\mathbb{U}}.
\end{equation}
The dual Problem $(D)$ associated with $(\hat{P}_2')$ is:
\begin{equation}
\label{dual problem formulation}
\begin{array}{ c c}
   (D)\quad  
   \underset{\lambda\in \mathbb{U}}{\max}\, \mathcal{W}(\lambda), 
   & \mbox{ where }\, \mathcal{W}(\lambda):=
   \underset{u\in\hat{\mathcal{U}} ,v\in \mathbb{U}}{\min} 
   L(u,v,\lambda).
\end{array}
\end{equation}
For any $\lambda\in \mathbb{U}$, it holds:
\begin{equation}
\label{expW}
\mathcal{W}(\lambda) = - F_0^\ast(\lambda) + \frac{1}{n}\sum_{i=1}^{n}\underset{u ^i\in\hat{\mathcal{U}}_i}{\min}\,\E [ G_i(u^i,\omega^i)]+\langle \lambda,\E \,u^i\rangle_{\mathbb{U}},
\end{equation}
where, for any real valued function $F$ defined on $\mathbb{U}$, $F^{\ast}$ stands for its Fenchel conjugate, defined for $x \in \mathbb{U}$ by $F^\ast(x):=\sup_{y\in \mathbb{U}}\, \langle x,y \rangle_{\mathbb{U}}-F(y)$.
The problem is said to be qualified if it is still feasible after a small perturbation of the constraint, in the following sense:
\begin{equation}
\label{cons qual}
\text{There exists $\varepsilon>0$ such that }
\mathcal{B}_{\mathbb{U}}(0,\varepsilon)
\subset g(\hat{\mathcal{U}},\mathbb{U}),
\end{equation}
where $\mathcal{B}_{\mathbb{U}}(0,\varepsilon)$ is the open ball of radius $\varepsilon$ in $\mathbb{U}$, $g$ has been defined in \eqref{def_g} {\color{black} and $g(\hat{\mathcal{U}},\mathbb{U})$ is the image by $g$ of $\hat{\mathcal{U}}\times \mathbb{U}$.}
\begin{lemma}
\label{constraint qualitfication}
If Assumption \ref{Assum_framework} holds, then Problem $(\hat{P}_2')$ is qualified. 
If Assumption \ref{F_0_conv} is also satisfied, then problems $(\hat{P}_2')$ and $(D)$ have the same value, the set of dual solutions $S$ is nonempty and bounded and any primal solution $\hat{u}$
satisfies both $W(\hat{\lambda}) = \tilde{J}(\hat{u})$ and $ (\hat{u},\hat{v})\in \underset{u\in\hat{\mathcal{U}},v\in \mathbb{U}}{\argmin}\,L(\hat{\lambda},u,v)$, with $\hat{\lambda}\in S$.
\end{lemma}

\begin{proof}
\if{}Choose $\varepsilon:=M$. Then: 
$\mathcal{B}_{\mathbb{U}}(0,\varepsilon)
\subset\overline{\mathcal{B}_{\mathbb{U}}(0,2M)}
=g(0,\mathbb{U}) \subset g(\mathcal{U},\mathbb{U})$, where $g(0,\mathbb{U})$ is the image by $g$ of $\{0\}\times\mathbb{U}$. Since $g(0,\mathbb{U}) = \mathbb{U}$}, the  qualification of $(\hat{P}_2')$ follows.
 \fi
By Assumption \ref{Assum_framework}.\ref{P1:feasible}, there exists $\check{u}$ feasible for Problem $(P_1)$. Then using the definition of $g$ in \eqref{def_g} 
\begin{equation}
\mathcal{B}_{\mathbb{U}}(0,\varepsilon) \subset \mathbb{U}
=  g(\check{u},\mathbb{U}) \subset g(\hat{\mathcal{U}},\mathbb{U}).
\end{equation}
The qualification of $(\hat{P}_2')$ follows.
The conclusion follows by \cite[Theorem 2.165]{bonnans2013perturbation}.
\end{proof}

Since the set of admissible controls  $\hat{\mathcal{U}}=\hat{\mathcal{U}}_{1}\times\ldots\times \hat{\mathcal{U}}_{n}$ is a Cartesian product, if $G_i$ is strictly convex with respect to its first variable,
then each component $\hat{u}^{ i}$ of the solution $\hat{u}$ of Problem $(\hat{P}_{2})$, can be uniquely determined by solving the following subproblem:
$$
\hat{u}^{i}=\underset{u^{i}\in\hat{\mathcal{U}}_i}{\argmin}\,
\left\{\mathbb{E}\left[G_{i}(u^{i},\omega^i)
+\langle\hat{\lambda},u^{i}\rangle_{ \mathbb{U}}\right]\right\},$$
where $\hat{\lambda}\in S$.
\begin{remark}
By using the same argument as in Theorem \ref{independence of ustar}, one can prove, for any $\lambda\in \mathbb{U}$:
\begin{equation}
 \underset{u^{i}\in\hat{\mathcal{U}_i}}{\min}\,
\left\{\mathbb{E}\left[G_{i}(u^{i},\omega^i)
+\langle\lambda,u^{i}\rangle_{ \mathbb{U}}\right]\right\}  
 =\underset{u^{i}\in\mathcal{U}_i}{\min}\,
\left\{\mathbb{E}\left[G_{i}(u^{i}(\cdot,\omega^{-i}),\omega^i)
+\langle\lambda,u^{i}\rangle_{ \mathbb{U}}\right]\right\}. 
\end{equation}
\end{remark}

\section{\textit{Stochastic Uzawa} and \textit{Sampled Stochastic Uzawa} algorithms}
\label{proposed algo}
This section aims at proposing an algorithm to find a solution of the dual problem \eqref{dual problem formulation}. 

\subsection{Preliminary results}

Though the below result is well-known and can be found in \cite{hiriart2004fundamentals} for functions defined on finite vector spaces, we adapt the proof to the Hilbert space setting for the sake of completeness.
\begin{lemma}
\label{strong conc fenchel F_0} 
If Assumption \ref{Assum_framework} and \ref{F_0_conv} are satisfied, then Assumption \ref{lip der assu} holds iff $F_{0}^{\ast}$ is
strongly convex.
\end{lemma}

\begin{proof}
(i)
Let Assumption \ref{lip der assu} holds.
Since $F_{0}$ is proper, convex and l.s.c., $F^{\ast}_{0}$ is l.s.c. proper. From the Lipschitz property of the gradient of $F_{0}$, it holds that $\mbox{dom}(F_{0})=\mathbb{U}$.

Let $s,\tilde{s}\in \mbox{dom}(F_{0}^{\ast})$ such that there exist $\lambda_{s}\in\partial F_{0}^{\ast}(s)$ and $\lambda_{\tilde{s}}\in\partial F_{0}^{\ast}(\tilde{s})$. From the differentiability, l.s.c. and convexity of $F_0$, it follows that: $s=\nabla F_{0}(\lambda_{s})$ and $\tilde{s}=\nabla F_{0}(\lambda_{\tilde{s}})$.
By Assumption \ref{lip der assu} and the extended Baillon-Haddad theorem  \cite[Theorem 3.1] {perez2021enhanced}, $\nabla F_{0}$ is cocoercive. In other words:
\begin{equation}
\begin{array}{ll}
 \langle s-\tilde{s},
\lambda_{s} -\lambda_{\tilde{s}} \rangle_{\mathbb{U}}
& =\langle \nabla F_{0}(\lambda_{s})-\nabla F_{0}(\lambda_{\tilde{s}}),
\lambda_{s} -\lambda_{\tilde{s}} \rangle_{\mathbb{U}}  \vspace{0.1cm}
\\
 & \geq \frac{1}{c}\Vert\nabla F_{0}(\lambda_{s})
 -\nabla F_{0}(\lambda_{\tilde{s}})\Vert_{\mathbb{U}}^{2} \vspace{0.1cm}
\\
&  =\frac{1}{c}\Vert s - \tilde{s} \Vert_{\mathbb{U}}^{2},
\end{array}
\end{equation}
where $c$ is the Lipschitz constant of $\nabla F_{0}$ defined in Assumption \ref{lip der assu}.
Therefore $\partial F_{0}^{\ast}$ is strongly monotone, which implies the strong convexity of $F_{0}^{\ast}$.
\\ (ii)
Conversely, assume that $F_{0}^{\ast}$ is proper and strongly convex.
Then there exist $\alpha,\beta>0$ and $\gamma\in \mathbb{U}$ such that for any $s\in \mbox{dom}(F^{\ast}_{0})$: $
 F_{0}^{\ast} (s) \geq \alpha\Vert s\Vert_{\mathbb{U}}^2  +\langle \gamma,\alpha\rangle_{\mathbb{U}}- \beta$,
 and $F_{0}$ being convex, l.s.c. and proper, for any $\lambda\in \mathbb{U}$ it holds: 
 \begin{equation}
 F_{0}(\lambda)\leq 
 \underset{s\in \mathbb{U}}{\sup}
 \langle s,\lambda-\gamma\rangle_{\mathbb{U}}
 - \alpha\Vert s\Vert^2_{\mathbb{U}}+\beta
 =
 \| \lambda-\gamma\|^2 / (4\alpha) + \beta.
 \end{equation}
Thus, $F_{0}$ is proper and uniformly upper bounded over bounded sets and therefore  is locally Lipschitz.
In addition, from the strong convexity of  $F_{0}^{\ast}$ and the convexity of $F_{0}$, for any $\lambda\in \mathbb{U}$, $\partial F_{0}(\lambda)$ is a singleton. Thus $F_0$ is everywhere Gâteaux differentiable. 

Let $\lambda,\mu\in \mathbb{U}$.
Since $F_{0}^{\ast}$ is strongly convex, the functions $F_{0}^{\ast}(s)-\langle \lambda,s\rangle_{\mathbb{U}}$ 
(resp. $F_{0}^{\ast}(s)-\langle \mu,s\rangle_{\mathbb{U}}$) has a unique minimum point $s_{\lambda}$ 
(resp. $s_{\mu}$), characterized by: $
\lambda\in\partial F_{0}^{\ast}(s_{\lambda})
\quad\mbox{and}\quad
\mu\in\partial F_{0}^{\ast}(s_{\mu})$. 
From the strong convexity of $F_{0}^{\ast}$, the strong monotonicity of $\partial F_{0}^{\ast}$ holds: $
\langle \mu-\lambda,
s_{\mu}-s_{\lambda}\rangle_{\mathbb{U}}
\geq 
\frac{1}{c}\,
\Vert s_{\mu}-s_{\lambda}\Vert^{2}_{\mathbb{U}}$, 
where $c>0$ is a constant related to the strong convexity of $F_{0}^{\ast}$. Using that $s_{\lambda}=\nabla F_{0}(\lambda)$ and $s_{\mu}=\nabla F_{0}(\mu)$, it holds:
\begin{equation}
\langle\, \mu - \lambda, 
\nabla F_{0}(\mu)- \nabla F_{0}(\lambda)\, \rangle_{L^{2}(0,T)}
\geq 
\frac{1}{c}\,
\Vert \nabla F_{0}(\mu)
- \nabla F_{0}(\lambda) \Vert_{L^{2}(0,T)}^{2},
\end{equation}
meaning that $\nabla F_{0}$ is cocoercive.
Applying the Cauchy–Schwarz inequality to the left hand side of the previous inequality, the Lipschitz property of $\nabla F_0$ follows.
\end{proof}

\begin{lemma}
\label{strong conc dual}
If Assumptions \ref{Assum_framework}, \ref{F_0_conv} and \ref{lip der assu} hold, then $\mathcal{W}$ is strongly concave.
\end{lemma}

\begin{proof}
For any $\lambda\in \mathbb{U}$, the expression of $\mathcal{W}(\lambda)$ is given by \eqref{expW},
where for any $i\in \{1,\ldots,n\}$,  $\lambda \mapsto \underset{u ^i\in\hat{\mathcal{U}}_i}{\inf}\,\E\,G_i(u^i,\omega^i)+\langle \lambda,\E\,u^i \rangle_{\mathbb{U}}$ is concave and, from Lemma \ref{strong conc fenchel F_0}, $-F_{0}^{\ast}$ is strongly concave. Since the sum of a concave function and of a strongly concave
function is strongly concave, the result follows.
\end{proof}

\subsection{Analysis of the algorithms}\label{analysis_algo}
Assumptions \ref{Assum_framework}, \ref{F_0_conv} and \ref{assumption_conv_algo} are supposed to hold throughout Section \ref{analysis_algo}.
For all $i\in\{1,\ldots n\}$, and $\lambda\in  \mathbb{U}$, we define the optimal control
$u^{i}(\lambda)$: 
\begin{equation}
\label{def u_lambda}
u^{i}(\lambda):=\underset{u^{i}\in\hat{\mathcal{U}}_{i}}
{\argmin}\,\left\{\mathbb{E}\left[G_{i}(u^{i},\omega^i)
+\langle\lambda,u^{i}\rangle_{\mathbb{U}}\right]\right\},
\end{equation} 
which is well defined since $u^{i}\to \mathbb{E}\,G_{i}(u^{i},\omega^i)$ is strictly convex.

For any $\lambda\in \mathbb{U}$, the subset $V(\lambda)$ is defined by: 

\begin{equation}
\label{def v_lambda}
V(\lambda) :=\underset{v\in  \mathbb{U}}{\argmin}\{F_{0}(v)
-\langle \lambda,v \rangle_{\mathbb{U}}\}.
\end{equation}
Since $F_{0}$ is convex and 
has at least quadratic growth,
$V(\lambda)$ is a non empty subset of $\mathbb{U}$ and is reduced to a singleton if $F_{0}$ is strictly convex. For any $\lambda\in \mathbb{U}$, we denote by $v(\lambda)$ an element of $V(\lambda)$, and for any $v(\lambda)\in V(\lambda)$, one has $v( \lambda)\in \partial F_0^\ast(\lambda)$.

Uzawa's algorithm seems particularly fitting for this problem. However, at each dual iteration $k$ and any $i\in\{1,\ldots , n\}$, one would have to compute the quantities $\mathbb{E}[u^{i}(\lambda^{k})]$ for the update of $\lambda^{k+1}$, which is hard in practice. Therefore two algorithms are proposed where, at each iteration $k$, $\lambda^{k+1}$ is updated thanks to a realization of $u^{i}(\lambda^{k})$.

We introduce the function $f\,:\,\mathbb{U}\to \mathbb{U}$ defined by:
\begin{equation}
\label{def f}
f(\lambda):=g(u(\lambda),v(\lambda))=\frac{1}{n}\sum_{i=1}^{n}\mathbb{E}\,u^{i}(\lambda)-v(\lambda).
\end{equation}
By Assumption \ref{assumption_conv_algo}.\ref{quad_growth}, $F_0$ has at least quadratic growth, whence $F_0^\ast$ has at most quadratic growth. Indeed, using the definition of the Fenchel's conjugate, the fact that $F_0$ has at least quadratic growth, and the Cauchy-Schwarz inequality, there exist $\bar{C}_1,\bar{C}_2>0$ such that for any $\lambda\in \mathbb{U}$:
\begin{equation*}
    F_0^\ast(\lambda)\leq \sup_{\mu \in \mathbb{U}}\,\Vert \mu \Vert_\mathbb{U}\Vert \lambda \Vert_\mathbb{U} - \bar{C}_1\Vert \mu \Vert_\mathbb{U}^2  + \bar{C}_2 = \frac{\Vert \lambda \Vert_\mathbb{U}^2}{2\bar{C}_1} +\bar{C}_2.
\end{equation*}
Then using  Lemma \ref{quad_growth_lemma} in Appendix \ref{append} and that $v(\lambda)\in \partial F_0^\ast(\lambda)$, there exists $C>0$ such that for any $\lambda\in \mathbb{U}$:
\begin{equation}\label{ineq_norm_v}
    \Vert v(\lambda)\Vert_\mathbb{U}\leq C(\Vert \lambda \Vert_\mathbb{U}+1).
\end{equation}
Using the definition of $\mathcal{U}$, one has $\textstyle \frac{1}{n}\sum_{i=1}^n\Vert \E\,u^i(\lambda) \Vert_\mathbb{U}\leq M$ for any $\lambda\in \mathbb{U}$. Therefore, from the definition of $f$ in \eqref{def f}, there exist $M_1,M_2>0$ such that for any  $\lambda\in \mathbb{U}$ one has: 
\begin{equation}
\label{f bounded}
\Vert f(\lambda)\Vert_{\mathbb{U}}^{2}\leq M_{1}+M_2\Vert \lambda \Vert_\mathbb{U}^2.
\end{equation} 
For any $\lambda \in  \mathbb{U}$, we denote by $\partial (-\mathcal{W}(\lambda))$ the subgradient of $-\mathcal{W}$ at $\lambda$. Therefore, for any $\lambda\in \mathbb{U}$:
\begin{equation}
\partial (-\mathcal{W}(\lambda))\ni -f(\lambda).
\label{theoreme enveloppe}
\end{equation}
The iterative algorithm, proposed as an approximation scheme for $\lambda^{*}\in\underset{\lambda}{\argmax}\, \mathcal{W}(\lambda)$, is summarized in the \textit{Stochastic Uzawa} Algorithm \ref{proposed algo desc}.

\begin{algorithm}[ht]
\caption{Stochastic Uzawa}
\label{proposed algo desc}
\begin{algorithmic}[1]
    \State Initialization $\lambda^{0}\in\mathbb{U}$, set $\{\rho_{k}\}_{k\in \mathbb{N}^\ast}$ satisfying Assumption \ref{assumption 1}.
    
    \State $k \leftarrow 0$.
    \For{$k=0,1,\ldots$}
    
        \State $v^{k} \leftarrow v(\lambda^{k})$ where $v(\lambda^{k})\in V(\lambda^{k})$ , this set being defined in \eqref{def v_lambda}.

        \State \label{update local}$u^{i,k} \leftarrow u^{i}(\lambda^{k})$ where $u^{i}(\lambda^{k})$ is defined in \eqref{def u_lambda} for any $i\in\{1,\ldots,n\}$.
        
        \State Generate $n$ independent noises  $(\omega^{1,k+1},\ldots,\omega^{n,k+1})$,  independent also of $\{\omega^{i,p}: 1\leq i\leq n, p\leq k\}$.
        
         \State \label{simulation state algo}Compute the associated control realization
         $(u^{1,k}(\omega^{1,k+1}),\ldots,u^{n,k}(\omega^{n,k+1}))$.
        
        \State \label{update Y} $ Y^{k+1}\leftarrow \frac{1}{n}\sum_{i=1}^{n}u^{i,k}(\omega^{i,k+1})-v^k$.

        \State $\lambda^{k+1}\leftarrow \lambda^{k}+\rho_{k}\,Y^{k+1}.$  \label{alt update price algo}

    \EndFor
\end{algorithmic}
\end{algorithm}

At any dual iteration $k$ of Algorithm \ref{proposed algo desc}, $Y^{k+1}$ is an estimator of  $\mathbb{E}\big[\frac{1}{n}\sum_{i=1}^{n}u^{i}(\lambda^{k})(\omega^{i,k+1})-v(\lambda^{k})\big]$.
An alternative approach, proposed in the \textit{Sampled Stochastic Uzawa} Algorithm \ref{alternative algo desc}, consists in performing less simulations at each iteration, by taking $m<n$, at the risk of performing more dual iterations, to estimate the quantity $\mathbb{E}\big[\frac{1}{n}\sum_{i=1}^{n}u^{i}(\lambda^{k})(\omega^{i,k+1})-v(\lambda^{k})\big]$.
\begin{algorithm}[ht]
\caption{Sampled Stochastic Uzawa}
\label{alternative algo desc}
\begin{algorithmic}[1]
    \State Initialization of $m$ a positive integer and $\check{\lambda}^{0}\in \mathbb{U}$, set $\{\rho_{k}\}_{k\in \mathbb{N}^\ast}$ satisfying Assumption \ref{assumption 1}.
    \State $k \leftarrow 0$.
    \For{$k=0,1,\ldots$}
    
        \State \label{opt v comp} $v^{k} \leftarrow v(\check{\lambda}^{k})$ where $v(\check{\lambda}^{k})\in V(\check{\lambda}^{k})$, this set being defined in \eqref{def v_lambda}.

        \State \label{I:def} Generate $m$ i.i.d. discrete random variables $I^k_{1},\ldots,I^k_{m}$ uniformly in $\{1,\ldots, n\}$.

        \State \label{opt u comp}
        $u^{I^k_{j},k} \leftarrow u^{I^k_{j}}(\check{\lambda}^{k})$ where $u^{I^k_{j}}(\check{\lambda}^{k})$ is defined in \eqref{def u_lambda} for any $j\in\{1,\ldots,m\}$.
        
         \State \label{W:def}Generate $m$ independent noises $(\omega^{1,k+1},\ldots,\omega^{m,k+1})$, independent also of $\{\omega^{i,p}: 1\leq i\leq m, p\leq k\}$.
         \State \label{alt simulation state algo}Compute the associated control realization 
         $(u^{I^k_{1},k}(\omega^{1,k+1}),\ldots,u^{I^k_{m},k}(\omega^{m,k+1}))$.
         
        \State \label{alt update Y} 
        $\check{Y}^{k+1}\leftarrow \frac{1}{m}\sum_{j=1}^{m}u^{I^k_{j},k}(\omega^{I^k_{j},k+1})-v^{k}$
        
        \State $\check{\lambda}^{k+1}\leftarrow \check{\lambda}^{k}+\rho_{k}\check{Y}^{k+1}.$  \label{update price algo}

    \EndFor
\end{algorithmic}
\end{algorithm}

 The complexity of the \textit{Sampled Stochastic Uzawa} Algorithm \ref{alternative algo desc} is proportional to $m\times K$, where $K$ is the total number of dual iterations and $m$ the number of simulations performed at each iteration. The error $\E\,\Vert\lambda^{k+1}-\lambda^*\Vert_{\mathbb{U}}^2$ for $\lambda^*\in S$ (we recall that $S$ is defined by $S:=\underset{\lambda\in   \mathbb{U}}{\argmax}\,\mathcal{W}(\lambda)$ and from Lemma \ref{constraint qualitfication} $S$ is non empty) is the sum of the square of the bias (which only depends on $K$ and not on $m$) and the variance (which both depends on $K$ and $m$). Therefore, this algorithm 
enables a bias variance trade-off for a given complexity. Similarly, for a given error, it enables to optimize the complexity of the algorithm.

The following result establishes the convergence of the \textit{Stochastic Uzawa} Algorithm \ref{proposed algo  desc}:

\begin{lemma}
\label{theorem convergence}
Let Assumptions \ref{Assum_framework}, \ref{F_0_conv}, \ref{assumption_conv_algo} and \ref{assumption 1} hold and let $\{\lambda_k\}_k$ be a sequence of multipliers generated by Algorithm \ref{proposed algo  desc}. Then:
\begin{enumerate}
[label={\normalfont (\roman*)}]
\item \label{conv norm lambda} $\{\Vert \lambda^{k}-\lambda \Vert^{2}_{  \mathbb{U}}\}$ 
converges a.s., for all $\lambda\in S$.
\item \label{conv W}$\mathcal{W}(\lambda^{k})\xrightarrow[k \to \infty]{}
\underset{\lambda\in   \mathbb{U}}
{\max} \,\mathcal{W}(\lambda)$  a.s.
\item \label{weak conv a.s lambda}$\{\lambda^{k}\}$ weakly converges to some $\bar{\lambda}\in S$ in $\mathbb{U}$ a.s.
\item \label{conv a.s lambda} If Assumption \ref{lip der assu} holds, then a.s. $\{\lambda^{k}\}$ converges to $\bar{\lambda}$ in $\mathbb{U}$, with $S:=\{\bar{\lambda}\}$.
\end{enumerate}
\end{lemma}

The proof follows from 
\cite[Theorem 3.6]{geiersbach2019projected}. The cited reference (changing minimization in maximization) is interested in the maximization of a function $\mathcal{W}$ of the specific form $ \mathcal{W}(\lambda)=\mathbb{E}\,\textbf{W}(\lambda,\omega)$, where $\textbf{W}(\cdot,\omega)$ is concave a.s. in $\omega$. However, in our setting we cannot in general exhibit such a representation for the dual function $\mathcal{W}$, defined in \eqref{dual problem formulation}. Using the definition of $u(\lambda)$ in \eqref{def u_lambda} and $v(\lambda)$ in \eqref{def v_lambda}, we have  $\mathcal{W}(\lambda)=\E\,\textbf{W}(\lambda,\omega)$, where
\begin{equation*}
\lambda \mapsto \textbf{W}(\lambda,\omega):=F_0(v(\lambda))+\frac{1}{n}\sum_{i=1}^nG_i(u^i(\lambda),\omega^i)+\langle\lambda,\frac{1}{n}\sum_{i=1}^{n}u^{i}(\lambda)-v(\lambda)\rangle_\mathbb{U}
\end{equation*}
Note tat $\textbf{W}(\cdot,\omega)$ is not a concave function of $\lambda$ for a.a. $\omega\in \Omega$. Although our setting does not enter in the framework considered in \cite{geiersbach2019projected}, the proof of Lemma \ref{theorem convergence} follows from an obvious adaptation of the one in \cite[Theorem 3.6]{geiersbach2019projected}. It is enough to provide the first steps of the proof.

\begin{proof}[Proof of Lemma \ref{theorem convergence}]
First consider point \ref{conv norm lambda}.
Let $\lambda\in S$.
For any $k$, $\mathcal{G}_{k+1}$ is the filtration defined by: 
         \begin{equation}
         \label{def:G}
         \mathcal{G}_{k+1} := \sigma\left(\{\omega^{i,p}\}: 1\leq i \leq n, \, p\leq k+1\}\right).
         \end{equation}
Using the definition of $Y^{k+1}\in \mathbb{U}$ line \ref{update Y} in the \textit{Stochastic Uzawa} Algorithm \ref{proposed algo desc}, we have:  
\begin{equation}
\begin{array}{ll}
\Vert \lambda^{k+1}-\lambda\Vert_{\mathbb{U}}^{2}
& =\Vert \lambda^{k} + \rho_{k}Y^{k+1} -\lambda \Vert_{\mathbb{U}}^{2} \\
 & = \Vert \lambda^{k}-\lambda\Vert_{\mathbb{U}}^{2} 
 +2\rho_{k}\langle \lambda^{k}-\lambda, Y^{k+1}\rangle_{  \mathbb{U}} +(\rho_{k})^{2}\Vert Y^{k+1} \Vert_{\mathbb{U}}^{2}.
\end{array}
\end{equation}
Since $Y^{k+1}$ is independent from $\mathcal{G}_{k}$, it follows that:
\begin{equation}
\mathbb{E}[\Vert Y^{k+1}\Vert_{  \mathbb{U}}^{2}|\mathcal{G}_{k}]
=\mathbb{E}\,
\Vert\frac{1}{n}\sum_{i=1}^{n}u^{i}(\lambda^{k})(\omega^{i,k+1})
-v(\lambda^{k})\Vert_{\mathbb{U}}^{2}.
\end{equation}
Using previous equality and the inequality \eqref{f bounded}, one can easily show that there exists $M_3,M_4>0$ such that, for any $k\in \mathbb{N}$, one has:
\begin{equation}
\label{ineq norm Y conditioned}
\mathbb{E}[\Vert Y^{k+1}\Vert_{ \mathbb{U}}^{2}|\mathcal{G}_{k}]
\leq M_1+M_2\Vert \lambda^k\Vert_\mathbb{U}^2\leq M_3+M_4\Vert \lambda^{k}-\lambda\Vert_{\mathbb{U}}^{2}
\end{equation}
Since $\lambda^{k}$ is $\mathcal{G}_{k}$-measurable and that $\mathbb{E}[Y^{k+1}|\mathcal{G}_{k}]=f(\lambda^{k})$, we have that:
\begin{equation}
\begin{array}{ll}
\mathbb{E}[\Vert \lambda^{k+1}-\lambda\Vert_{\mathbb{U}}^{2}|\mathcal{G}_{k}]  
 &= \Vert \lambda^{k}-\lambda\Vert_{\mathbb{U}}^{2}
 +2\rho_{k}\mathbb{E}[\langle \lambda^{k}-\lambda,Y^{k+1}\rangle_\mathbb{U}|\mathcal{G}_{k})]
 +(\rho_{k})^{2}\mathbb{E}[\Vert Y^{k+1}\Vert_{  \mathbb{U}}^{2}|\mathcal{G}_{k}]\\
&\leq \Vert \lambda^{k}-\lambda\Vert_{\mathbb{U}}^{2} 
+2\rho_{k}\langle \lambda^{k}-\lambda,f(\lambda^{k})\rangle_\mathbb{U}
+(\rho_{k})^{2}( M_3+M_4\Vert \lambda^{k}-\lambda\Vert_{\mathbb{U}}^{2})\\
&\leq \Vert \lambda^{k}-\lambda\Vert_{\mathbb{U}}^{2}(1+M_4\rho_k^2)
+ (\rho_{k})^{2}M_{3}-2\rho_{k}(\mathcal{W}(\lambda)-\mathcal{W}(\lambda^{k})).
\end{array}
\label{serie of ineq}
\end{equation}
In the last inequality, we used the concavity of $\mathcal{W}$ and \eqref{theoreme enveloppe}. The rest of the proof follows
\cite[Theorem 3.6]{geiersbach2019projected}.
\end{proof}

Recalling the definition of $\bar{J}(u,v)$ in \eqref{eq:P2v} and of $\bar{\lambda}$ in Lemma \ref{theorem convergence}.\ref{weak conv a.s lambda}, we define $\bar{u}$:
\begin{equation}
\label{def u bar}
\bar{u} :=\underset{u\in\hat{\mathcal{U}}}
{\argmin}\,\left\{\mathbb{E}\left[\sum _{i=1}^{n}G_{i}(u^{i},\omega^i)
+\langle\bar{\lambda},u^{i}\rangle_{\mathbb{U}}\right]\right\}.
\end{equation}
Under Assumption \ref{assumption_conv_algo},
$G_i$ is strictly convex w.r.t. the first variable, and then $\bar{u}$ is well defined.
If $F_{0}$ is strictly convex, then $V(\bar{\lambda})$ is a singleton and we can write:
\begin{equation}
\label{def v bar}
\bar{v} :=\underset{v\in\mathbb{U}}
{\argmin}\,\left\{F_{0}(v)
+\langle\bar{\lambda},v\rangle_{\mathbb{U}}\right\}.
\end{equation}

\begin{remark}\label{saddel_point_lag}
If $F_0$ is convex, by Lemma \ref{constraint qualitfication}, there is no duality gap associated to the Lagrangian $L$ defined in \eqref{def_lag}. 
Further, if $F_0$ is strictly convex, then $(\bar{u},\bar{v},\bar{\lambda})$ is the unique saddle point associated to the Lagrangian $L$. Indeed, by  Assumption \ref{lip der assu}, $\bar{\lambda}$ is the unique solution of the dual problem $(D)$, by Assumption \ref{assumption_conv_algo}.\ref{strict_gi}, $\bar{u}$ is unique and by strict convexity of $F_0$, $\bar{v}$
is also the unique minimizer in the right hand side of \eqref{def v bar}.

\end{remark} 
\begin{theorem}
\label{result of liminf conv}
Let the Assumptions  \ref{Assum_framework}, \ref{F_0_conv}, \ref{lip der assu}, \ref{assumption_conv_algo} and \ref{assumption 1} hold, then we have:
\begin{enumerate}
[label={\normalfont (\roman*)}]
\item
\label{weak conv u}
 $\{u(\lambda^{k})\}$ weakly converges a.s. to  $\bar{u}$.
 \end{enumerate}
Furthermore, if $F_{0}$ is strictly convex,
then from Remark \ref{saddel_point_lag}, $\bar{u}$ is the unique minimizer of $\tilde{J}$ in $\hat{\mathcal{U}}$ and:
\begin{enumerate}
\item[{\normalfont (ii)}]
\label{conv J_tilde}
$\tilde{J}(u(\lambda^{k}))\xrightarrow[k\to \infty]{}
\tilde{J}(\bar{u})\, \mbox{ a.s.}$ 
\item[{\normalfont (iii)}]
\label{conv J}
$\underset{k\to\infty}{\lim \sup}\,J(u(\lambda^{k}))\leq \inf_{u\in\mathcal{U}} J(u)+2\,\varepsilon$ a.s. where $\varepsilon = cM^{2}/n$.
\end{enumerate}
\end{theorem}
\begin{proof}
Proof of point \ref{weak conv u}.
By Lemma \ref{theorem convergence}.\ref{conv a.s lambda},
the sequence $\{\lambda^k\}$ is bounded in $\mathbb{U}$. Thus, using inequality \eqref{ineq_norm_v} one deduces that $\{v(\lambda^k)\}$ is also bounded in $\mathbb{U}$. Since the sequence $ \{(u(\lambda^{k}),v(\lambda^{k}))\}$ is bounded in $\hat{\mathcal{U}}\times \mathbb{U}$, there exists a weakly convergent subsequence $ \{(u(\lambda^{\theta_{k}}),v(\lambda^{\theta_{k}}))\} $ such that:
\begin{equation}
\label{weak conv u v}
(u(\lambda^{\theta_{k}}),v(\lambda^{\theta_{k}}))
\underset{k\to \infty} {\rightharpoonup}
(u^{\theta},v^{\theta})\in \hat{\mathcal{U}}\times\mathbb{U}.
\end{equation}
Using the definition of $\lambda\mapsto u(\lambda)$ in \eqref{def u_lambda}, it holds for any $k>0$:
\begin{equation}
\label{ineq for th 5.3}
 \mathbb{E}\left[G_{i}(\bar{u}^{i},\omega^i))
+\langle\lambda^{\theta_k},\bar{u}^{i})
\rangle_{\mathbb{U}}\right]   
\geq
\mathbb{E}\left[G_{i}(u^{i}(\lambda^{\theta_k}),\omega^i)
+\langle\lambda^{\theta_k},u^{i}(\lambda^{\theta_k})
\rangle_{ \mathbb{U}}\right)].
\end{equation}
Using that  $u^i\mapsto  G_i(u^i,\omega^i)$ is a.s.
 w.l.s.c. on $\hat{\mathcal{U}}_i$
and the a.s. convergence of $\{\lambda^{k}\}$, resulting from Lemma \ref{theorem convergence}.\ref{conv a.s lambda}, we have from \eqref{ineq for th 5.3} when $k\to\infty$ :
\begin{equation}
\label{ineq u weak limit}
\mathbb{E}\left[G_{i}(\bar{u}^{i},\omega^i)
+\langle\bar{\lambda},\bar{u}^{i})
\rangle_{\mathbb{U}}\right]
\geq
\mathbb{E}\left[G_{i}(u^{i,\theta},\omega^i)
+\langle\bar {\lambda},u^{i,\theta}
\rangle_{\mathbb{U}}\right].
\end{equation}
Since $\bar{u}$ is unique, it follows $u^{\theta}=\bar{u}$ and \eqref{ineq u weak limit} is an equality.
Using that every weakly convergent subsequence of $\{u(\lambda^{k})\}$ has the same weak limit $\bar{u}$, \ref{weak conv u} is deduced.

Proof of point (ii).

From point \ref{weak conv u} and \eqref{ineq u weak limit}, it follows for any $i\in\{1,\ldots,n\}$:
\begin{equation}
\label{equation F_i weak limit}
\underset{k\to \infty}{\lim}
\mathbb{E}\,G_{i}(u^{i}(\lambda^{k}),\omega^i)
= 
\mathbb{E}\,G_{i}(\bar{u}^{i},\omega^i).
\end{equation}
Using \eqref{weak conv u v}, the w.l.s.c. of $F_{0}$, equation \eqref{def v bar}, and applying the same previous argument to $\{v(\lambda^{k})\}$, it holds that:
\begin{equation}
\underset{k\to\infty}{\lim}F_{0}(v(\lambda^{k}))
-\langle \lambda^{k},v(\lambda^{k})
\rangle_{\mathbb{U}}
=
F_{0}(\bar{v})
-\langle \bar{\lambda},\bar{v}
\rangle_{\mathbb{U}},
\end{equation}
and $v(\lambda^{k})\underset{k\to\infty}{\rightharpoonup}\bar{v}$.

From the two previous equalities and the a.s. convergence of $\{\lambda^{k}\}$, it follows:
\begin{equation}
\label{conv F0 vk}
\underset{k\to\infty}{\lim}F_{0}(v(\lambda^{k}))
=
F_{0}(\bar{v}).
\end{equation}
Using that $(\bar{u}, \bar{v},\bar{\lambda})$ is a saddle point, it follows:
\begin{equation}
\label{constraint satisfied}
\frac{1}{n}\sum_{i=1}^{n}\E\,\bar{u}^i
=\bar{v}.
\end{equation}
From \eqref{conv F0 vk} and \eqref{constraint satisfied}, it holds:
\begin{equation}
\label{equation F_0 weak limit}
\underset{k\to\infty}{\lim}F_{0}\left(\frac{1}{n}\sum_{i=1}^{n}\E\,u^i(\lambda^k)\right)
=
F_{0}\left(\frac{1}{n}\sum_{i=1}^{n}\E\,\bar{u}^i \right).
\end{equation}
Then adding \eqref{equation F_i weak limit} and \eqref{equation F_0 weak limit}: $\underset{k\to\infty}{\lim}\tilde{J}(u(\lambda^k))=\tilde{J}(\bar{u})$.

Proof of point (iii).
From point (ii), inequality \eqref{ineq th epsilon} and Theorem \ref{optimality th}, it holds:
\begin{equation}
\underset{k\to\infty}{\lim\sup}\,{J}(u(\lambda^k))
\leq \underset{k\to\infty}{\lim\sup}\,\tilde{J}(u(\lambda^k))+ \varepsilon
=\inf_{u\in\mathcal{U}}\tilde{J}(u) + \varepsilon
\leq  \inf_{u\in\mathcal{U}} J(u)+2\varepsilon.
\end{equation}
\end{proof}

\begin{lemma}
\label{v lip}
Let Assumptions \ref{Assum_framework} and \ref{strong conv hyp}.\ref{strong conv F0} hold, then the function $\lambda\mapsto v(\lambda)$ is Lipschitz on $\mathbb{U}$.
\end{lemma}

\begin{proof}
From the definition of $v$ in \eqref{def v_lambda}, we have for any $\lambda\in \mathbb{U}$ that $
\lambda\in\partial F_{0}(v(\lambda))$.
Thus, for any $\lambda,\mu \in \mathbb{U}$, we have from the strong convexity of $F_{0}$:
\begin{equation}
\left \{
\begin{array}{ll}
 F_{0}(v(\mu)) &  \geq F_{0}(v(\lambda)) + \langle \lambda , v(\mu)-v(\lambda)\rangle_{\mathbb{U}}
 +\alpha\Vert v(\mu)-v(\lambda)\Vert^{2}_{\mathbb{U}} \\
  F_{0}(v(\lambda)) &  \geq F_{0}(v(\mu))
  + \langle \mu , v(\lambda)-v(\mu)\rangle_{\mathbb{U}}
  +\alpha\Vert v(\lambda)-v(\mu)\Vert^{2}_{\mathbb{U}} .
\end{array}
\right.
\end{equation}
Adding the two previous inequalities, after simplications, we get:
\begin{equation}
\langle \lambda - \mu, v(\lambda)-v(\mu)\rangle_{\mathbb{U}} \geq 2 \alpha\Vert v(\lambda)-v(\mu)\Vert^{2}_{\mathbb{U}} .
\end{equation}
Applying the Cauchy-Schwarz inequality and simplifying by $\Vert v(\lambda)-v(\mu)\Vert_{\mathbb{U}}$, we get the desired Lipschitz inequality.
\end{proof}

\begin{lemma}
\label{lip u lambda}
If Assumptions \ref{Assum_framework} and \ref{strong conv hyp}.\ref{strong conv Fi} hold, the function $\lambda\mapsto u(\lambda)$ is Lipschitz on $\mathbb{U}$.
\end{lemma}

\begin{proof}
The proof is similar to the proof of Lemma \ref{v lip}.
\end{proof}

\begin{theorem}
\label{conv to opt sol}
Let Assumptions \ref{Assum_framework}, \ref{lip der assu},
\ref{assumption 1}, and \ref{strong conv hyp}  hold, then:
 $u(\lambda^{k})\xrightarrow[k\to\infty]{}u(\bar{\lambda})\, \mbox{ a.s.}$
\end{theorem}

\begin{proof}
The convergence follows from the Lipschitz property of $\lambda\mapsto u(\lambda)$ (as a result 
of Assumption \ref{strong conv hyp}) associated with the a.s. convergence of $\{\lambda^{k}\}$.
\end{proof}

\begin{remark}
Note that Lemma \ref{theorem convergence} and Theorems \ref{result of liminf conv} and \ref{conv to opt sol} still hold when replacing  $\lambda^k$ by $\check{\lambda}^k$ and $Y^k$ by 
$\check{Y}^k$ (as defined respectively in line \ref{alt update Y} and \ref{update price algo} of the \textit{Sampled Stochastic Uzawa} Algorithm \ref{alternative algo desc}). This can be proved by the same argument, using that $\check{Y}^k$ is bounded a.s. and $\E[\check{Y}^k\vert \check{\mathcal{G}}_k]=f(\check{\lambda}^k)$ for any $k$, where:
\begin{equation}
\check{\mathcal{G}}_k 
=
\sigma\Big(\{W^{I^{p}_\ell,p}\}: 1\leq \ell \leq m, \, p\leq k\}\Big)
\vee
\sigma\Big(\{I^{p}_\ell\}: 1\leq \ell \leq m, \, p\leq k\}\Big),
\end{equation}
with $W^{I^{k}_\ell,p}$ and  $I^{p}_\ell$ defined respectively in line \ref{W:def} and \ref{I:def} of the \textit{Sampled Stochastic Uzawa} Algorithm \ref{alternative algo desc}.
\end{remark}

\begin{remark}
We wish to emphasize that the proposed Algorithm \ref{alternative algo desc} is particularly suitable for practical distributed implementations in agent-based scenarios. In these cases, the quantity $\lambda$ can be interpreted as a common price signal that is broadcast to the independent agents, which in turn compute independently their optimal solution $u(\lambda)$ on the basis of their local parameters (step \ref{update local}).
\end{remark}
To illustrate the results, we consider in the next section stochastic control problems in both continuous and discrete time settings.
\section{Application to stochastic control}\label{examples_sto_control}

\subsection{Continuous time setting}
Let $(\Omega, \mathcal{F},\mathbb{F},\mathbb{P})$ be a complete filtered probability space on which $W=(W^{i})_{i=1,\ldots,n}$ is a $n\times d-$dimensional Brownian motion such that, for any $t\in[0,T]$ and $i\in\{1,\ldots,n\}$, $W^{i}_{t}$ takes value in $\mathbb{R}^{d}$ and generates the filtration $\mathcal{F}=(\mathcal{F}_{t})_{0\leq t\leq T}$. In the considered notation, $\mathbb{P}$ stands for the Wiener measure associated with this filtration and $\mathbb{F}$ for the augmented filtration by all $\mathbb{P}$-null sets. The following notations are used:
\begin{equation*}
\mathbb{X} :=  \{\varphi :\Omega
\to \mathcal{C}([0,T],\mathbb{R}^{d}) \,
|\,\varphi(\cdot)\,\mbox{is}\,\mathbb{F}-\mbox{adapted}, \,\Vert \varphi\Vert_{\infty,2}:=\mathbb{E}
\underset{ 
\substack{1\leq k\leq d \\ s\in [0,T]}}
{\sup}|\varphi_k(s)|^{2}<\infty\},
\end{equation*}
\begin{equation*}
\U:=L^{2}([0,T],\mathbb{R}^{p}):=  \{\varphi  :[0,T]
\to \mathbb{R}^{p}  \,
|\,\int_{0}^{T}\sum_{k=1}^p|\varphi_k(t)|^{2}dt<\infty\},
\end{equation*}
For any $i\in\{1,\ldots,n\}$, the feasible set of controls is defined by:
\begin{equation}
\begin{array}{ll}
\mathcal{U}_{i}:= &\{v:\Omega\times [0,T]\to \mathbb{R},v(\cdot)\mbox{ is }\mathbb{F}-\mbox{prog. measurable},  \\
 & v(\omega)\in \mathbb{U}\mbox{ and }v_{t}(\omega)\in[-M_{i},M_{i}]^p, \mbox{ for a.a. }(t,\omega)\in [0,T]\times \Omega\},
\end{array}\vspace{0.1 cm}
\label{def of U}
\end{equation}
and we set $M:=\underset{i\in\{1,\ldots,n\}}{\max}\, M_{i}$, where $M_i>0$.

Each local agent $i=1,\ldots,n$ is supposed to control its state variable through the control process $u^{i}\in\mathcal{U}_{i}$ and is subject to independent uncertainties. More specifically, the state process of each agent, $X^{i,u^{i}}=(X^{i,u^{i}}_{t})_{t\in[0;T]}$, for $i = 1,\ldots,n$ takes values in $\mathbb{R}^{d}$ and follows the dynamics  for $i\in\{1,\ldots,n\}$:
\begin{equation}
    \left\{
    \begin{array}{l l l }
    dX^{i,u^{i}}_{t}&=&\mu_{i}(t,u^{i}_{t}(\cdot,W^i),X^{i,u^{i}}_{t})dt+\sigma_{i}(t,X^{i,u^{i}}_{t})dW_{t}^{i}, \mbox{ for }t\in[0,T], \\
    X_ {0}^{i,u^{i}} & = & x_{0}^{i}\in\mathbb{R}^{d};
    \end{array}
    \right.
    \label{evolution SDE}
\end{equation}
We assume that, for any $i$, there exist five functions  $\alpha_i\in L^{\infty}([0,T],\mathbb{R}^{d\times p})$, $\beta_i,\theta_i\in L^{\infty}([0,T],\mathbb{R}^{d\times d}), \gamma_i\in L^{\infty}([0,T],\mathbb{R}^{d})$ and $\xi_i \in L^{\infty}([0,T],\mathbb{R}^{d\times d\times d})$ such that, for any $(t,\nu,x)\in[0,T]\times [-M,M]^p\times \mathbb{R}^d$:
\begin{equation}
\label{linear dynamic}
\mu_{i}(t,\nu,x)
=\alpha_{i}(t)\nu+\beta_{i}(t)x+\gamma_{i}(t)
\,\mbox{ and }
\,\sigma_{i}(x,t)=\xi_{i}(t)x+\theta_{i}(t).
\end{equation}
Without loss of generality, the initial states $x_{0}^{i}$ are supposed to be deterministic.
The process $X^{i,u^i}$ is $\mathbb{F}$-progressively measurable.
For all $i$, $\mathcal{F}^{i}$ stands for the natural filtration of the Brownian motion $W^{i}$.

\subsubsection{On the well-posedness of $(P_1)$}

In this section, we discuss some conditions under which Problem $(P_{1})$ is well-posed.
\begin{lemma}
\label{lem:X}
Let $i\in\{1,\ldots,n\}$ and $v\in\mathcal{U}_{i}$ be a control process.
The map $v\mapsto X^{i,v}$ is linear continuous from $\mathcal{U}_i$ to $\mathbb{X}$ and
there exists a unique process $X^{i,v}\in \X$ satisfying  \eqref{evolution SDE} (in the strong sense) such that, for any $p\in[1,\infty)$:
\begin{equation}
\label{eq:supXp}
\E\big[\sup_{\substack{0\leq t\leq T\\1\leq k\leq d}}
\vert X_{k,t}^{i,v}\vert^r\big]
<C(r,T,x_0,K)<\infty \ .
\end{equation}
\end{lemma}

\begin{proof}
The proof for the existence and uniqueness of a solution of \eqref{evolution SDE} relies on \cite[Theorem 3.6, Chapter 2]{mao2007stochastic}. The inequality is a result of \cite[Theorem 4.4, Chapter 2 ]{mao2007stochastic}. 
\end{proof}
Let $F_{0}: \U 
\rightarrow \mathbb{R}$ be a proper, convex and lower semi continuous  function, satisfying Assumptions \ref{lip der assu} and \ref{assumption_conv_algo}.\ref{quad_growth}.
For any $i\in\{1,\ldots,n\}$, we assume that there exists $F_i$ such that the local cost $G_i$ is of the form: 
\begin{equation}
\label{def_g_sto_cont}
u^i\mapsto G_i(u^i(\cdot,\omega^{-i}),\omega^i)=F_i(u^i(\omega^i,\omega^{-i}), X^{i,u^i}(\omega^i)),
\end{equation}
where $F_{i}:\U \times \mathcal{C}([0,T]\times \mathbb{R}^d)\rightarrow \mathbb{R}_+$ is a proper and lower semi continuous function. Additional assumptions are formulated below.
\begin{assumption}
\label{assum functions}
For any $i\in\{1,\ldots,n\}$:
\begin{enumerate}
[label={\normalfont (\roman*)}]
\item \label{hyp 221} {\color{black}$F_{i}$ is jointly convex w.r.t. both variables and strictly convex w.r.t. the first variable}. 
\item \label{hyp 222} there exists a positive integer $r$ such that
$F_{i}$ has $r$-polynomial growth, i.e there exists $K>0$ such that for any $x^i\in \mathcal{C}([0,T],\mathbb{R}^{d})$ and $u^i\in \U $: $\vert F_i(u^i,x^i)\vert \leq K (1+\sup_{\substack{0\leq t\leq T \\ 0\leq k \leq n}} \vert x^i_{k,t}\vert ^r)$.
\end{enumerate}
\end{assumption}

Note that if Assumption \ref{assum functions}.\ref{hyp 221} is satisfied, using the definition of $F_0$ and $F_i$ above in this section, Assumptions \ref{Assum_framework} to \ref{assumption_conv_algo} hold.
\begin{remark}
 It is worth highlighting the following aspects regarding Assumption \ref{assum functions}:
\begin{enumerate}
\item Assumption \ref{assum functions}.\ref{hyp 221} is satisfied if there exist $g_{i}:L^{2}((0,T),\mathbb{R}^p)\to \mathbb{R}$ strictly convex and $h_{i}:\mathcal{C}([0,T],\mathbb{R}^d)\to \mathbb{R}$ convex, such that $F_{i}(v,X)=g_{i}(v)+h_{i}(X)$.
\item Observe that Assumption \ref{assum functions} satisfies Assumptions \ref{Assum_framework}.\ref{convexity:Gi} and  \ref{assumption_conv_algo}.\ref{strict_gi}
\end{enumerate}
\end{remark}

From now on, Assumption \ref{assum functions} is in force. Now the optimization problems  $(P_{1}^{c})$ and 
$(\hat{P}_{2}^{c})$
can be clearly defined:          
\begin{equation}
(P_{1}^{c})\left\{
\begin{array}{l}
{ \inf_{u\in\mathcal{U}}J^{c}(u)}\\
{ J^{c}(u):= \E\left[F_{0}(\frac{1}{n}\sum_{i=1}^{n}u^i(\omega))+\frac{1}{n}\sum_{i=1}^{n}F_{i}(u^i(\omega),X^{i,u^i}(\omega^i))\right]},
\end{array}
\right.
\label{P1 c}
\end{equation}
and
\begin{equation}
(\hat{P}_{2}^{c})
\left\{
\begin{array}{l}
{ \inf_{u\in \hat{\mathcal{U}}
}\tilde{J}^{c}(u)}\\
{ \tilde{J}^{c}(u):= F_{0}\left(\frac{1}{n}\sum_{i=1}^{n}\E\,u^i\right)+\frac{1}{n}\E\left[\sum_{i=1}^{n}F_{i}(u^i(\omega),X^{i,u^i}(\omega^i))\right]},
\end{array}
\right.
\label{P2 c}
\end{equation}
Using the results of Section \ref{approx_results}, we can state the following corollary:
\begin{corollary}
\begin{enumerate}[label=(\roman*)]
\item \label{well posedness continuous} Problems $(P_{1}^{c})$ and
$(\hat{P}_{2}^{c})$ admit both a unique solution.
\item \label{eps opt con}Any optimal solution of Problem
$(\hat{P}_{2}^{c})$ is an $\varepsilon$-optimal solution, where $\varepsilon=cM^{2}/n$, of Problem $(P_{1}^{c})$. 
\end{enumerate}
\end{corollary}

\begin{proof}
The proof of point \ref{well posedness continuous} is a specific case of Lemma \ref{existence solution P1}.
Similarly, point \ref{eps opt con} is a particular case of Theorem \ref{optimality th}.
\end{proof}

\begin{remark}
A practical example of this type of stochastic optimization problem is illustrated in Section \ref{example illustration}, which considers the interactions between a large population of price-responsive self-interested domestic appliances and a central system operator which has to meet the prescribed levels of demand at minimum generation costs. 
\end{remark}

\subsection{Discrete time setting}
\label{subsection_discrete}
 The main results of the paper are instantiated to the discrete time setting in this subsection. 
The following notations are used.
\begin{itemize}
    \item Let  $n\in\mathbb{N}^{\ast}$ be the number of agents, $d,p\in\mathbb{N}^{\ast}$ the dimension respectively of their state and control variables at any time step, and $T\in\mathbb{N}^{\ast}$ the finite time  horizon.
    \item For any matrix $M$,  its transpose is denoted by 
    $M^{\top}$.
    \item We consider a global noise process as a sequence of independent random variables $(W_{1},\ldots,W_{T})$, where   for any $t\in\{1,\ldots,T\}$, $W_{t}$ is a vector of $d$-dimensional 
    independent random variables, with finite variance,
    defined on the probability space $(\Omega,\mathcal{F},\mathbb{P})$: $
        W_{t}:=(W_{t}^{1}, \ldots, W_{t}^{n})$, with $W_t^i\in \mathbb{R}^d$.
    For any $i\in\{1,\ldots,n\}$ and $t\in\{1,\ldots,T\}$ we define $\mathcal{F}^{i}_{t}:=\sigma(W_{1}^{i},\ldots,W_{t}^{i})$ and  $\mathcal{F}_{t}:=\otimes_{i=1}^{n}\mathcal{F}_{t}^i$.
    \item The space $\X$ is defined by: \begin{equation}
     \X:= \{x=(x_{0},\ldots,x_{T})\,\vert \,
    \forall t\in \{0,\ldots, T\},\mathbb{R}^d\ni 
    x_t\mbox{ is }
    \mathcal{F}_t-\mbox{measurable and }\E\Vert x_t\Vert_2^2<\infty\}.
    \end{equation}

    \item For any $i\in\{1,\ldots,n\}$, we define the space of control $\mathcal{U}^{i}$ of agent $i$ by:
    \begin{equation}
    \begin{array}{ll}
     \mathcal{U}^{i}:=& \{u^i=(u^i_{0},\ldots,u^i_{T-1})\,\vert \,
    \forall t\in \{0,\ldots, T-1\},\mathbb{R}^p\ni 
    u^i_t\mbox{ is }\\
    &\mathcal{F}_t-\mbox{measurable and }u_k^i(\omega)\in[-M,M]^p\quad \mathbb{P}\mbox{-a.s.} \},
    \end{array}
    \end{equation}
    where $M>0$.
   We finally set $\mathcal{U}:=\prod_{i=1}^{n}\mathcal{U}^{i}$.

    \item  For any $i\in\{1,\ldots,n\}$, $X^{i,u^{i}}:=(X^{i,u^i }_{0},\ldots,X^{i,u^i}_{T})\in \X$ is the state trajectory of agent $i$ controlled by $u^{i}\in \mathcal{U}^i$.  
    We have the following dynamics:
    \begin{equation}
    \left\{
    \begin{array}{l l l l}
    X_{t+1}^{i,u^i}&=& A^iX_{t}^{i,u^i}+B^iu^i_{t}+W^i_{t+1},
    & \mbox{for }t\in\{0,\ldots,T-1\} ,\\
    X_{0}^{i,u^i}     & = & x_{0}\in\mathbb{R}^{d},&\\
    \end{array}
    \right.
    \label{disc linear SDE}
    \end{equation}
    where $A^i\in \mathbb{R}^{d\times d}$ and  $B^i\in \mathbb{R}^{d\times p}$.
 \end{itemize}
 Let  $F_{0}\,:\,\mathbb{R}^{p\times T}\to \bar{\mathbb{R}}$ be proper, lower semi continuous, convex and satisfy Assumptions \ref{lip der assu} and \ref{assumption_conv_algo}.\ref{quad_growth}. Similarly to the previous subsection, we assume that, for any $i$, there exists  a function $F_{i}\,:\,\mathbb{R}^{p\times T}\times\mathbb{R}^{d\times T}\to \mathbb{R}$ such that $G_i$ and $F_i$ satisfy \eqref{def_g_sto_cont}, and $F_i$ satisfies Assumption \ref{assum functions} for an integer $r$ such that $\E\,\Vert W_t \Vert^r$ is finite for any $t\in \{1,\ldots, T\}$.
 
Now, for any $n\in\mathbb{T}^{\ast}$, the optimization problems  $(P_{1}^{d})$ and $(\hat{P}_{2}^{d})$
can be clearly defined:          
\begin{equation}
(P_{1}^{d})\left\{
\begin{array}{l}
{ \inf_{u\in\mathcal{U}}J^{d}(u)}\\
{ J^{d}(u):= \E\left[F_{0}(\frac{1}{n}\sum_{i=1}^{n}u^i)+\frac{1}{n}\sum_{i=1}^{n}F_{i}(u^i,X^{i,u^i})\right]},
\end{array}
\right.
\label{P1 h}
\end{equation}
and
\begin{equation}
(\hat{P}_{2}^{d})
\left\{
\begin{array}{l}
{ \inf_{u\in \hat{\mathcal{U}}
}\tilde{J}^{d}(u)}\\
{ \tilde{J}^{d}(u):= F_{0}\left(\frac{1}{n}\sum_{i=1}^{n}\E\,u^i\right)+\frac{1}{n}\E\left[\sum_{i=1}^{n}F_{i}(u^i,X^{i,u^i})\right]}.
\end{array}
\right.
\label{P2 h}
\end{equation}
In the same spirit as in the previous subsection, we have the following results, which will be useful for the next section.
\begin{corollary}
\begin{enumerate}[label=(\roman*)]
\item \label{well posedness discrete} Problems $(P_{1}^{d})$ and $(\hat{P}_{2}^{d})$
admit both a unique solution.
\item \label{eps opt dis}Any optimal solution of Problem $(\hat{P}_{2}^{d})$
is an $\varepsilon$-optimal solution, where $\varepsilon=cM^{2}/n$, of Problem $(P_{1}^{d})$. 
\end{enumerate}
\end{corollary}

\begin{proof}
The proof of point \ref{well posedness discrete} is analogous to the one of Lemma \ref{existence solution P1}.
Similarly, proof of point \ref{eps opt dis} is analogous to the one of Theorem \ref{optimality th}.
\end{proof}
One can implement the \textit{Stochastic Uzawa} (Algo \ref{proposed algo desc}) and the \textit{Sampled Stochastic Uzawa} (Algo \ref{alternative algo desc}) in this discrete time setting with Lemma  \ref{theorem convergence} and Theorems \ref{result of liminf conv} and \ref{conv to opt sol} still ensuring the algorithm convergence. 

\section{A numerical example: the LQG (Linear Quadratic Gaussian) problem}

This section aims at illustrating numerically the convergence of the \textit{Stochastic Uzawa} (Algo \ref{proposed algo  desc}) on a simple example. The speed of convergence of the algorithm is evaluated according to the number of dual iterations and of agents. A linear quadratic formulation is considered, with $n$ agents in a discrete setting Problem $(\hat{P}_{2}^{LQG})$. 
We use the notations of Section \ref{subsection_discrete}. 

This framework constitutes a simple test case, since the (deterministic) Uzawa's algorithm can be performed, and one can compare the resulting multiplier estimate  
with the one provided by the \textit{Stochastic Uzawa} algorithm. All the assumptions required for the convergence of the \textit{Stochastic Uzawa} (Algo \ref{proposed algo  desc}) are satisfied for Problem  $(\hat{P}_2^{LQG})$.
Moreover, the optimal solutions to the local problems (line \ref{update local} of Algorithm \ref{proposed algo desc}) can be resolved analytically.

Problem  $(\hat{P}_{2}^{LQG})$ 
is similar to $(\hat{P}_2^d)$
defined in \eqref{P2 h} but, in this specific case, the function $F_0$ is a quadratic function of the aggregate strategies of the agents
    \begin{equation}
    F_0\left(\frac{1}{n}\sum_{i=1}^n \E \,u^i\right) :=\frac{\nu}{2} \sum_{t=0}^{T}\left(\frac{1}{n}\sum_{i=1}^n \E\,u^i_{t}-r_{t}\right)^{2},
    \end{equation}
    where $\nu>0$ and $\{r_{t}\}$ is a deterministic target sequence. Similarly, the cost  term $F_i$ of the individual agents is expressed as a quadratic function of their state $X^{i,u^{i}}$ and control $u^{i}$
    \begin{equation}
    F_{i}(u^{i},X^{i,u^{i}}):=\frac{1}{2}\left(\sum_{t=0}^{T}
    d_{i}({X_{t}^{i,u^{i}}})^{2}
    + q_{i}(u_{t}^{i})^{2}\right)
    +\frac{d^{f}_{i}}{2}({X_{T}^{i,u^{i}}})^{2},
    \end{equation}
where $q_{i}>0$ and $d_{i}>0$ for any $i\in\{1,\ldots,n\}$. Defining the matrices $D=\mbox{diag}(d_{1},\ldots,d_{n})$, $Q=\mbox{diag}(q_{1},\ldots,q_{n})$ and $D^{f}=\mbox{diag}(d_{1}^{f},\ldots,d_{n}^{f})$, we get:
\begin{equation}
\sum_{i=1}^n F_{i}(u^{i},X^{i,u^{i}}) =\frac{1}{2}\left(
\sum_{t=0}^{T}
{X_{t}^{u \top}}DX_{t}^{u}
+ {u_{t}^{\top}}Qu_{t}\right)
+\frac{1}{2}{X_{T}^{u \top}}D^{f}X_{T}^{u},
\end{equation}
where, for any $t\in\{0,\ldots,T\}$,
    $X_{t}^{u}:=(X^{1,u^1}_{t},\ldots,X^{n,u^n}_{t})\in\mathbb{R}^{ n}$ is the controlled state vector of all the agents.
Now the optimization Problem $(\hat{P}_{2}^{LQG})$
is clearly defined. 

To find the optimal multiplier and control of $(\hat{P}_{2}^{LQG})$,
the \textit{Stochastic Uzawa} Algorithm \ref{proposed algo  desc} is applied. In this specific case, the lines \ref{opt v comp} and \ref{opt u comp} take respectively the following form at any dual iteration $k$:
\begin{equation}
\label{local:prob}
u^{i}(\lambda^{k}):=\underset{u^{i}\in \hat{U}^{i}}{\argmin}\,\left\{\mathbb{E}
\left[\frac{1}{2}\big(\sum_{t=0}^{T}
    d_{i}({X_{t}^{i,u^{i}}})^{2}
    + q_{i}(u_{t}^{i})^{2}+\lambda_{t}^{k}u_{t}^{i}\big)
    +\frac{d^{f}_{i}}{2}({X_{T}^{i,u^{i}}})^{2} \right]\right\},
\end{equation}
\begin{equation}
v(\lambda^{k}):=\underset{v\in \mathbb{R}^{T}}{\argmin}
\left\{(\sum_{t=0}^{T}
    \nu\, (v_{t}-r_{t})^{2} - \lambda_{t}^{k}v_{t}
    \right\}.
\end{equation}
The optimization problem \eqref{local:prob} solved by each local agent also falls within the LQG framework. One can solve  these problems using the results of \cite{todorov2006optimal}. The resolution via Riccati equations of \eqref{local:prob} shows that $u^{i}(\lambda^{k})$ is a linear function of the state $X^{i,u^i}$ and of the price $\lambda^k$. Therefore, in this specific example, one can explicitly compute { $\E[u_t^{i}(\lambda^k)\vert \mathcal{G}_{k}]$ for any $t$, with $\mathcal{G}_{k}$ as defined in \eqref{def:G}.

Within this described framework, it is possible to implement the (deterministic) Uzawa's algorithm and use it as a reference to evaluate the performances of the \textit{Stochastic Uzawa} algorithm.

Different population sizes $n$ are considered, with $n$ ranging between $1$ and $10^4$. Similarly, the algorithm is stopped after different numbers of dual iteration $k$, ranging between $10$ and $10^4$. In order to evaluate the bias and variance of the \textit{Stochastic Uzawa} algorithm, this has been  performed over $J=1000$ runs.

It is possible to define a Problem $(\hat{P}_2'^{LQG})$ and a dual Problem $(D^{LQG})$ from Problem $(\hat{P}_2^{LQG})$ following the same approach presented  in \eqref{eq:P2v} and \eqref{dual problem formulation} for the definition of $(\hat{P}_2')$ and $(D)$, respectively,  from $(\hat{P}_2)$.
It can be shown that, for any $n$, there exists a unique optimal multiplier $\bar{\lambda}^{n}$, solution of $(D^{LQG})$.}
For any $n$, the quantity $\lambda^{k,n,j}$ denotes the dual price computed during the $j^{th}$ simulations ($j=1,\ldots,J$) of the \textit{Stochastic Uzawa} algorithm, after $k$ dual iterations.

For any $n$, the deterministic multiplier $\bar{\lambda}^{n}$ is obtained by applying Uzawa's algorithm, after $10^4$ dual iterations. To this end, we applied the \textit{Stochastic Uzawa} Algorithm \ref{proposed algo  desc}, where we ignored the line \ref{update Y} and we replaced the update of $\lambda^{k}$ line \ref{alt update price algo} by: $
\bar{\lambda}^{k+1}\leftarrow \bar{\lambda}^k + \rho_k ( \frac{1}{n}\sum_{i=1}^{n} \E\,u^i(\bar{\lambda}^k)-v(\bar{\lambda}^k))$.

At each dual iteration $k$, the computation of $\E\,u^i(\lambda^k)$ is straightforward in this specific case, since $u^i(\lambda^k)$ is a linear function of $X^{i,u^i}$ and $\lambda^k$, as explained in the previous subsection.

The multipliers $\lambda^{k,n,j}$ and $\bar{\lambda}^n$, obtained by applying the \textit{Stochastic Uzawa} and \textit{Uzawa algorithms}, respectively, are now compared.
For any $k$ and $n$,  let $b_{k,n}$, $v_{k,n}$ and $\ell_{k,n}$
denote an estimation of the bias, the variance and the $L^2$ norm of the error, respectively, as computed via Monte Carlo method with $J$ simulations.
For any $k$ and $n$, these quantities are defined as follows: 
\begin{align*}
    b_{k,n}&:=\frac{1}{J}\sum_{j=1}^{J}
 \lambda^{k,n,j}-\bar{\lambda}^{n},
\\
 v_{k,n} &:= \frac{1}{J}\sum_{j=1}^{J} \Vert  \lambda^{k,n,j}-\bar{\lambda}^{n} - b_{k,n} \Vert_{2}^{2},
\\
\ell_{k,n} &:= v_{k,n} + \Vert b_{k,n} \Vert_{2}^{2}.
\end{align*}

\begin{figure*}
\centering
\begin{minipage}[t]{.30\linewidth}
\centering
\pgfplotsset{xmin=0.8, xmax=4.2, ymin=-4.8, ymax=5.5,
every axis title shift={6 pt}}
\begin{tikzpicture}[font=\tiny]
\begin{axis}
[legend entries=
{$n= 10$, $n= 10^{2}$ , $n= 10^{3}$,$n= 10^{4}$},
legend style={at={(1,1)},anchor=north east},legend columns=2,
grid=major,height=5cm,width=4.8cm,
  xtick={1,2,3,4},
  ytick={-4,-2,...,4},
xlabel={$\log_{10}(k)$},
title={}, cycle list name=black white]
\node at (240,630) {slope\,$\simeq -0.8$};

\addplot table[x index=0,y index=2]{log_variance_price_x_iteration.txt};

\addplot table[x index=0,y index=3]{log_variance_price_x_iteration.txt};

\addplot table[x index=0,y index=4]{log_variance_price_x_iteration.txt};

\addplot table[x index=0,y index=6]{log_variance_price_x_iteration.txt};

\end{axis}
\end{tikzpicture}
\caption{ Variance term $\log_{10}(v_{k,n})$, expressed as a function of $k$, for different number of agents $n$.}
\label{variance price it study iteration}
\end{minipage}\hspace*{0.1cm}
\begin{minipage}[t]{.30\linewidth}
\centering
\pgfplotsset{xmin=-0.2, xmax=4.2, ymin=-4.8, ymax=6,
every axis title shift={6pt}}
\begin{tikzpicture}[font=\tiny]
\begin{axis}[legend entries=
{$k=10$, $k=10^{2}$, $k=10^{3}$, $k=10^{4}$},
legend style={at={(1,1)},anchor=north east},legend columns=2,
grid=major,,height=5cm,width=4.8cm,
  xtick={0,1,2,3,4},
  ytick={-4,-2,...,4},
xlabel={$\log_{10}(n)$},
title={}, cycle list name=black white]
\node  at (325,630)  {slope\,$\simeq - 1$};
\addplot  table[x index=0,y index=1]{log_variance_price_x_population.txt};

\addplot table[x index=0,y index=3]{log_variance_price_x_population.txt};

\addplot table[x index=0,y index=5]{log_variance_price_x_population.txt};

\addplot table[x index=0,y index=7]{log_variance_price_x_population.txt};

\end{axis}
\end{tikzpicture}
\caption{ Variance term $\log_{10}(v_{k,n})$, expressed as a function of $n$, for different number of iterations $k$.}
\label{variance price it study population}
\end{minipage}\hspace*{0.1cm}
\begin{minipage}[t]{.30\linewidth}
\centering
\begin{tikzpicture}[font=\tiny]
\begin{axis}[legend entries={ $n= 10^{4}$},legend style={at={(1,1)},anchor=north east},legend columns=1,
grid=major,height=5cm,width=4.8cm,
xlabel={$\log_{10}(k)$},
title={}, cycle list name=black white]
\addplot table[x index=0,y index=6]{log_bias_price_x_iteration.txt};
\end{axis}
\end{tikzpicture}
\caption{ Bias term $\log_{10}(\Vert b_{k,n} \Vert_{2}^{2})$, expressed as a function of $k$, given the number of agents $n=10^4$.}
\label{bias price it study iteration}
\end{minipage}
\end{figure*}

Since numerical simulations are based on finite dimensional approximations, it is relevant to compare the empirical convergence rates, shown in Figures \ref{variance price it study iteration}-\ref{bias price it study iteration},  with the associated theoretical asymptotic rates presented in the literature for a finite dimensional setting.

In Figure \ref{variance price it study iteration}, we observe a behavior in $1/k^{\alpha}$ (with $\alpha\simeq 0.8$) of the variance $v_{k,n}$  w.r.t. the number of iterations $k$. This rate of convergence is consistent with  \cite[Theorem 2.2.12, Chapter 2]{duflo2013random}, 
where the best asymptotic convergence rate for the Robbins-Monro algorithm is proved to be of the order of $1/k$ (for the quadratic error).

In Figure \ref{variance price it study population}  we observe a behavior in $1/n^{\beta}$ (with $\beta\simeq 1$) of the variance $v_{k,n}$  w.r.t. the number of agents $n$. This is expected, following \cite[Theorem 2.2.12, Chapter 2]{duflo2013random} and the observation that the variance of $Y^{k+1}$ is {of the order of} $1/n$ for any iteration $k$.

Finally, in Figure \ref{bias price it study iteration}, we note that the bias $\Vert b_{k,n}\Vert^{2}$ decreases faster  than $1/k$ w.r.t. the number of iterations $k$. Thus, for a large number of iterations ($k>0$), the dominant term impacting the error $l_{k,n}$ is the variance $v_{k,n}$.

\section{Price-based coordination of a large population of thermostatically controlled loads} 
\label{example illustration}

The goal of this section is to
demonstrate 
the applicability of the presented approach for the coordination of TCLs 
in the context of flexible power systems.
In particular, the problem analyses the daily operation of a power system with a large penetration of price-responsive TCLs, adopting a modelling framework similar to \cite{depaola2019mean}. Two distinct elements are considered: i) a system operator, which must schedule a portfolio of generation assets in order to satisfy the energy demand at a minimum cost, and ii) a population of price-responsive TCLs that individually determine their ON/OFF power profile in response to price, 
with the objective of minimizing their operating cost while fulfilling users’ requirements. 
Note that the operations of the two elements are interconnected, since the aggregate power consumption of the TCLs will modify the system-level demand profile that needs to be accommodated by the system operator.

\subsection{Formulation of the problem}

In the considered problem, the function $F_0$ represents the minimized power production cost and corresponds to the resolution of an Unit Commitment (UC) problem. The UC determines generation scheduling decisions (in terms of energy production and frequency response (FR) provision) in order to minimize the short term operating cost of the system while matching generation and demand. The demand quantity is the sum of an inflexible deterministic term (denoted for any time instant $t\in[0,T]$ by $\bar{D}(t)$) and of a stochastic component $n\times U_{TCL}(t)$, i.e. the product of the population size $n$ and the average demand profile $U_{TCL}(t)$ of the TCL population.

For simplicity, a Quadratic Programming (QP) formulation in a discrete time setting is adopted for the UC problem. The central planner disposes of $Z$ generation technologies (gas, nuclear, wind) and schedules their production and allocates response by slot of $30$ min every day. For any $j\in\{1,\ldots,Z\}$ and $\ell\in\{1,\ldots,48\}$, the quantities $H_{j}(t_{\ell})$, $G_{j}(t_{\ell})$ and $R_{j}(t_{\ell})$ denote respectively the commitment, the power production and the frequency-response from unit $j$ during the time interval $[t_{\ell},t_{\ell+1}]$ (all expressed in MWh). The associated vectors are denoted by $H(t_{\ell})=[H_{1}(t_{\ell}),\ldots,H_{Z}(t_{\ell})]$,  $G(t_{\ell})=[G_{1}(t_{\ell}),\ldots,G_{Z}(t_{\ell})]$ and $R(t_{\ell})=[R_{1}(t_{\ell}),\ldots,R_{Z}(t_{\ell})]$.

The cost sustained at time $t_\ell $ by unit $j$ is linear with respect to the commitment $H_j(t_\ell)$ and quadratic with respect to generation $G_j (t_l)$ and can be expressed as $c_{1,j}H_j (t_\ell)G^{Max}_{j}(t_\ell)+ c_{2,j}G_j (t_\ell) + c_{3,j}G_j (t_\ell)^2$. In this cost expression, $G^{Max}_{j}$ denotes the production limit allocated by each generation technology, $c_{1,j}$ [\textsterling/MWh] is the no-load cost term, whereas $c_{2,j}$ [\textsterling/MWh] and $c_{3,j}$ [\textsterling/MW$^2$h] are the production cost coefficients of the generation technology $j$.
The optimization of $F_0$  must satisfy the following constraints for all $\ell\in\{1,\ldots,48\}$ and $j\in\{1,\ldots,Z\}$:
\begin{equation}
\label{equ offer demand}
\sum_{j=1}^{Z}G_{j}(t_{\ell})-\int_{t_{\ell}}^{t_{\ell+1}}(\bar{D}(t)+n\,U_{TCL}(t))dt=0,
\end{equation}
\begin{equation}
\label{commit:bounds}
0\leq H_j(t_{\ell})\leq 1,
\end{equation}
\begin{equation}
\label{ineq prod max prod_1}
R_j(t_\ell)-r_j H_j(t_\ell)G_j^{max}(t_\ell)\leq 0,
\end{equation}
\begin{equation}
\label{ineq prod max prod_2}
R_j(t_\ell)-s_j(H_j(t_\ell)G_j^{max}(t_\ell)-G_j(t_\ell))\leq 0,
\end{equation}
\begin{equation}
\label{2e}
\Delta G_L-\Lambda\left(\bar{D}(t_\ell )+n(\bar{U}_{TCL}(t_\ell)-\bar{R}_{TCL}(t_\ell)\right)\Delta f_{qss}^{max} - \hat{R}(t_\ell)\leq 0,
\end{equation}
\begin{equation}
\label{2f}
2\Delta G_L t_{ref}t_d - t_{ref}^2\hat{R}(t_\ell) - 4\Delta f_{ref}t_d\hat{H}(_\ell)\leq 0,
\end{equation}
\begin{equation}
\label{2g}
\bar{q}(t) - \hat{H}(_\ell)\hat{R}(_\ell)\leq 0\,
\end{equation}
\begin{equation}
\label{2h}
\mu\, r_j H_j(t_\ell)G^{max}_j(t_\ell )-G_j(t_\ell)\leq 0,
\end{equation}
where \eqref{equ offer demand} equals production and aggregated demand (i.e. the system inelastic demand $\bar{D}$ and the TCL flexible demand $n U_{TCL}$).
The quantities $\hat{R}$ and $\hat{H}$ denote the total reserve and inertia of the system, respectively, and are defined for any $\ell \in \{ 1,\dots,48 \}$ as: 
\begin{align*}
    \hat{R}(t_\ell) & = \sum_{j=1}^{Z}R_j(t_\ell)+nR_{TCL}(t_\ell),
    \\
    \hat{H}(t_\ell) &= \sum_{j=1}^{Z}\frac{h_jH_j(t_\ell)G_j^{max}-h_L\Delta G_L}{f_0}.
\end{align*}
 In \eqref{commit:bounds} it is supposed that, for any generation technology $j$, the capacity of the single power plant is significantly smaller than the total installed capacity. As a result, it is reasonable to consider the continuous relaxation of the UC problem by assuming $H_j(t_\ell)\in[0,1]$.

The amount of response allocated by each generation technology is limited by the
headroom $r_j H_j (t_\ell)G^{max}_{j}(t_\ell)$ in \eqref{ineq prod max prod_1} and by the slope $s_j$ linking the FR with the dispatch level \eqref{ineq prod max prod_2}.
Constraints \eqref{2e} to \eqref{2h} deal with frequency response provision and $R_{TCL}$ (the mean of FR allocated by TCLs). They guarantee secure frequency deviations following sudden generation loss $\Delta G_L$. Inequality \eqref{2e} allocates enough FR (with delivery time $t_d$) such that the quasi-steady-state frequency remains above $\Delta f_{qss}^{max}$, with $\Lambda$ accounting for the damping effect introduced by the loads \cite{kundur1994power}. The constraint \eqref{2g} imposes the maximum
tolerable frequency deviation $\Delta f_{nad}$, following the formulation
and methodology presented in \cite{teng2015stochastic} and \cite{trovato2018unit}.
The rate of change of frequency is taken into account in \eqref{2f} where at $t_{rcf}$ the frequency deviation remains above $\Delta f_{ref}$.
Constraint \eqref{2h} prevents trivial unrealistic
solutions that may arise in the proposed formulation, such as high values of committed generation $H_j(t_\ell)$
in correspondence with low (even zero) generation dispatch
$G_j(t_\ell)$. The reader can refer to \cite{depaola2019mean} for more details on the UC problem.

The solution $F_0$ of the UC problem
 can be defined by the following optimization problem:
\begin{equation}
F_{0}(U_{TCL},R_{TCL}):=\underset{H,G,R }{\min}\sum_{\ell=1}^{48}
\sum_{j=1}^{Z} c_{1,j}H_{j}(t_{\ell})G_j^{max}(t_\ell)+c_{2,j}G_{j}(t_{\ell})+c_{3,j}G_{j}(t_{\ell})^{2},
\end{equation}
subject to equations \eqref{equ offer demand}-\eqref{2h}.

Note that the formulation of the present problem does not fulfill all the assumptions presented in Section \ref{proposed algo}. In particular, the function $F_0$ is not strictly convex, as instead supposed in Theorem \ref{result of liminf conv}.$($ii$)$.$($iii$)$. Nevertheless, the numerical simulations of Section \ref{dec:impl} shows that the proposed approach is still able to achieve convergence.

Regarding the modelling of the individual price-responsive TCLs, each TCL $i\in\{1,\ldots,n\}$ is characterized at any time $t\in[0,T]$ by its temperature state $X^{i,u^{i}}_{t}$ $[^{\circ}C]$ and by its power consumption control $u^{i}_{t}$  $[W]$. The thermal dynamic $X^{i,u^{i}}_{t}$ of a single TCL $i$ is given by:
\begin{equation}
    \label{temp dyn}
    \left\{
    \begin{array}{l l l  l}
    dX^{i,u^{i}}_{t}&=&-\frac{1}{\gamma_{i}}(X_{t}^{i,u^{i}}-X_{OFF}^{i}
    +\zeta_{i} u^{i}_{t})dt
    +\sigma_{i} \,dW_{t}^{i}, & \mbox{for }t\in[0,T], \\
    X_{0,u^{i}}^{i}     & = & x_{0}^{i}\in\mathbb{R},& 
    \end{array}
    \right.
\end{equation}
where:
\begin{itemize}
\item $\gamma_{i}$ is its thermal time constant $[s]$.
\item $X_{OFF}^{i}$ is the ambient temperature $[^{\circ}C]$.
\item $\zeta_{i}$ is the heat exchange parameter $[^{\circ}C/W]$.
\item $\sigma_{i}$ is a positive constant  $[({^{\circ}C})s^{\frac{1}{2}}]$,
\item $W^{i}$ is a Brownian Motion $[s^{\frac{1}{2}}]$, independent from $W^{j}$ for any $j\neq i$.
\end{itemize}

For any $i\in\{1,\ldots,n\}$, the set of control $\mathcal{U}_{i}$ is defined by:
\begin{equation}
\begin{array}{ll}
\mathcal{U}_{i}:= &\{v:\Omega\times [0,T]\to \mathbb{R},v(\cdot)\mbox{ is }\mathbb{F}-\mbox{prog. measurable},  \\
 & v(\omega)\in \mathbb{U}\mbox{ and }v_{t}(\omega)\in\{0,P_{ON,i}\}, \mbox{ for a.a. }(t,\omega)\in [0,T]\times \Omega\},
\end{array}\vspace{0.1 cm}
\end{equation}
The TCLs dynamics in \eqref{temp dyn} have been derived according to \cite{kizilkale2019integral}, with the addition of the stochastic term $\sigma_i dW_{t}^{i}$ to account for the influence of the environment (opening/closing of the fridge, environment temperature, etc.) on the evolution of the TCL temperature. 

By combining the objective functions of the systems, the system operator
 has to solve the following optimization problem:
\begin{equation}
\label{central planner pb}
\begin{array}{r l}
   (P_{1}^{TCL})  &  
   \left\{
\begin{array}{l}
{\displaystyle \inf_{u\in\mathcal{U}} J(u)}\\
\begin{array}{ll}
 J(u):=
 &\E\left[F_0\left(\frac{1}{n}\sum_{i=1}^n u^i,\frac{1}{n}\sum_{i=1}^n r_i(u^i, X^{i,u^i}) \right) \right]\\
 & +\E\left[\frac{1}{n}\sum_{i=1}^n
\int_{0}^{T}f_{i}(u^{i}_{s},X_{s}^{i,u^{i}})ds 
+\gamma_{i}(X^{i,u^{i}}_{T}-\bar{X}^{i})^{2}\right],
\end{array}
\end{array}
\right .
\end{array}
\end{equation}
The term $r_i(u^i, X^{i,u^i})$ denotes the maximum amount of FR allocated by the TCL $i$ at time $s$ and can be expressed as:
\begin{equation}
r_i(u^i, X^{i,u^i})(s):=u^i_s \frac{X^{i,u^i}_s-X^i_{min}}{X^i_{max}-X^i_{min}}.
\end{equation}
The discomfort term of the single TCL $i$ at time $s$ is denoted by $f_{i}(u^{i}_{s},X_{s}^{i,u^{i}})$, which takes the following expression:
\begin{equation}
\label{f_i:def}
 f_{i}(u^{i}_{s},X_{s}^{i,u^{i}})
 :=  \alpha_{i}\,(X^{i,u^{i}}_{s}-\bar{X}^{i})^{2}
  + \beta_{i}((X_{\min}^{i}-X^{i,u^{i}}_{s})_{+}^{2}+ (X^{i,u^{i}}_{s}-X_{\max}^{i})_{+}^{2}),
 \end{equation}
 where:
 \begin{itemize}
\item $\alpha_{i}(X^{i,u^{i}}_{s}-\bar{X}^{i})^{2}$ is a discomfort term penalizing temperature deviations from some comfort target $\bar{X}$ $[^{\circ}C]$, considering $\alpha_{i} \, [\mbox{\textsterling }/h(^{\circ}C)^{2}]$  as a discomfort term parameter.
\item $\beta_{i}((X^{i,u^{i}}_{s}-X_{\min}^{i})_{+}^{2}
 + (X_{\max}^{i}-X^{i,u^{i}}_{s})_{+}^{2})$ is a penalization term meant to maintain the temperature within the interval $[X_{\min}^{i},X_{\max}^{i}]$, considering the cost parameter $\beta_i$ $[\mbox{\textsterling }/s(^{\circ}C)^{2}]$ and the maximum function $(a)_{+}=\max(0,a)$.
\item $\gamma_{i}(X^{i,u^{i}}_{T}-\bar{X}_{i})^{2}$ is a terminal cost term meant to impose soft periodic constraints by quadratically penalizing the deviations of the final temperature state $X^{i,u^{i}}_{T}$ with respect to the initial temperature value $\bar{X}_{i}$, considering the cost parameter $\gamma$  $[\mbox{\textsterling }/s(^{\circ}C)^{2}]$.
\end{itemize}

Note that the control set $\mathcal{U}$ is not convex. We can mention a possible relaxation of the problem by taking the control in the interval $[0,P_{ON,i}]$.

In order to solve $(P_{1}^{TCL})$, the modified Problem $(P_{2}^{TCL})$ is studied:
\begin{equation}
\label{central planner modified pb}
\begin{array}{r l}
   (P_{2}^{TCL})  &  
   \left\{
\begin{array}{l}
{\displaystyle \inf_{u\in\mathcal{U}} \tilde{J}(u)}\\
{\displaystyle 
\begin{array}{ll}
 \tilde{J}(u):=
 &F_0\left(\frac{1}{n}\sum_{i=1}^n \E\,u^i,\frac{1}{n}\sum_{i=1}^n \E\,r_i(u^i, X^{i,u^i}) \right)
 \\& +\E\left[\frac{1}{n}\sum_{i=1}^n
\int_{0}^{T}f_{i}(u^{i}_{s},X_{s}^{i,u^{i}})ds
+\gamma_{i}(X^{i,u^{i}}_{T}-\bar{X}^{i})^{2}\right].
\end{array}
}
\end{array}
\right .
\end{array}
\end{equation}

\subsection{Decentralized implementation}
\label{dec:impl}

The \textit{Sampled Stochastic Uzawa} Algorithm \ref{alternative algo desc} is applied to solve $(P_{2}^{TCL})$, with $m=317$ simulations per iteration. At each iteration $k$, the lines \ref{opt v comp} and \ref{opt u comp}  of Algorithm \ref{alternative algo desc} correspond to the solution of a deterministic UC problem and of an Hamilton Jacobi Bellman (HJB) equation, respectively. The time steps $\Delta t=7.6$ s and temperature steps $\Delta T=0.15^{\circ}C$ are chosen for the discretization of the HJB equation. Let us note that, at line \ref{opt u comp}, each TCL solves its own local problem on the basis of the received price signal $\lambda^k=(p^k,\rho^k)$:
\begin{equation}
\underset{u^i\in \mathcal{U}_i}{\inf}\int_0^T f_i(u^i_s,X_s^{i,u^i})+u^i_sp^k_s-r_i(u^i,X^{i,u^i})(s)\rho^k_s \,ds,
\end{equation}
where $f_i(u^i_s,X_s^{i,u^i})$ is a discomfort term defined in \eqref{f_i:def}, $u^i_sp^k_s$ can be interpreted as consumption cost and $r_i(u^i,X^{i,u^i})(s)\rho^k_s$  as fee awarded for FR provision. This implementation has a practical sense: each TCL uses local information and the received price signals to schedule its power consumption on the time interval $[0,T]$, with the objective of minimizing its overall costs.
It follows that, with the proposed approach, it is possible to optimize the total system costs in ($P_1^{TCL}$) in a distributed manner, with each TCL acting independently and pursuing its own cost minimization.

\subsection{Results}
In the proposed case study, the considered generation technologies available in the system are nuclear, combined cycle gas turbines (CCGT), open cycle gas turbines (OCGT)
and wind. The characteristics and parameters of the UC in this simulation are the same as in \cite{depaola2019mean}. 

It is assumed that the population of TCLs corresponds to $n = 2\times 10^7$ fridges with
built-in freeze compartment that operate in the system according
to the proposed price-based control scheme. 
For any TCL $i$, we set the consumption parameter $P_{ON,i}=180W$. The values of the TCL dynamic parameters $\gamma_i$ and $X^i_{OFF}$ of \eqref{temp dyn} are equal to the ones considered in \cite{depaola2019mean}. 
The initial temperatures of the TCLS are selected randomly according to a uniform probability distribution, considering temperature values between $-21^{\circ}C$ and $-14^{\circ}C$.
For any TCL $i$, the parameters of the individual cost function $f_i$, defined in \eqref{f_i:def}, are: $\alpha_i = 0.2\times 10^{-4}$ \textsterling /s$(^{\circ}C)^2$, $\beta_i = 50$\textsterling /s$(^{\circ}C)^2$, $\bar{X}^i = -17.5^{\circ}C$ and $X_{max}= -14^{\circ}C$, $X_{min}= -21^{\circ}C$. The parameter $\beta_i$ is intentionally taken very large to ensure that the TCL temperature remains within the interval $[X_{max}^i,X_{min}^i]$. Note that the individual problems solved by the TCLs are distinct than the ones in \cite{depaola2019mean} (different terms and parameters).

Simulations are performed for different volatility values $\sigma_i:=0,1,2$ (all the TCLs have the same volatility in the simulations), with $\sigma_i$ defined as in \eqref{temp dyn}. The \textit{Sampled Stochastic Uzawa} Algorithm is stopped after 75 iterations.

The resulting profiles of total power consumption $U=n\, U_{TCL}$  and total allocated response $R=nR_{TCL}$ by the TCLs population are reported in Figure \ref{conso comp}, while the resulting electricity prices $p$ and response availability prices $\rho$  are shown in Figure \ref{price comp}.
As observed in \cite{depaola2019mean}, the total consumption $U$ is higher when the electricity price $p$ is lower. Conversely, the total allocated response $R$ is higher when the FR remuneration price $\rho$ is also higher. This can be observed in particular during the first hours of the day, between $0$ and $6$ h. The
power consumption $U$ exhibits smaller oscillations during the rest of the day, as the internal temperature of the TCLs is maintained within feasible levels.
Although the prices do not seem to be particularly sensitive with respect to the volatility parameter $\sigma$, the power consumption $U$ and frequency response $R$ are highly correlated to the volatility of the TCLs temperature.

The TCLs impact on system commitment decisions and consequent energy/FR dispatch levels is also analyzed and displayed  in Figures \ref{prod comp} and \ref{fr comp}. 
 In this analysis, the ``flexibility scenario'', obtained with the proposed optimization strategy and considering flexible price-responsive TCLs, is compared to a ``business-as-usual'' scenario where the TCL do not respond to external price signals and do not perform any optimization of their costs. In the ``business-as-usual'', we impose $R_{TCL}(t) = 0$ and we assume that the TCLs operate exclusively according to their internal temperature $X^{i,u^i}$. They switch ON ($u^i(t) = P_{ON,i}$)
when they reach their maximum feasible temperature
$X^i_{max}$ and they switch back OFF again ($u^i(t) = 0$) when
they reach the minimum temperature $X^i_{min}$. 
In Figure \ref{prod comp}, we can clearly observe that TCL's flexibility allows to increase the contribution of wind to the energy balance of the system while decreasing the contribution of CCGT both in energy and frequency response.  In the ``business-as-usual'' scenario, without frequency support by the TCL, the optimal solution envisages a further curtailment of wind output in favor of an increase in CCGT
generation, as wind does not provide any FR. As expected, the influence of the TCL on the system is larger when the temperature volatility is lower. 
\begin{center}
\begin{figure}
\begin{minipage}[t]{.45\linewidth}

\begin{tikzpicture}[font=\scriptsize]

\begin{axis}[ymax = 1350,
legend entries={$U_{\sigma = 0}$,$U_{\sigma = 1}$,$U_{\sigma = 2}$, $R_{\sigma = 0}$,$R_{\sigma = 1}$,$R_{\sigma = 2}$},legend style={at={(1,1)},anchor=north east},legend columns=3,
grid=major,height=7cm,width=7cm,
xlabel={Time (h)} 
]

\addplot [thick, draw=blue] table[x index=0,y index=1]{U.txt};

\addplot [thick, draw=red] table[x index=0,y index=2]{U.txt};

\addplot [thick, draw=black] table[x index=0,y index=3]{U.txt};

\addplot [thick, loosely dashed ,draw=blue] table[x index=0,y index=1]{R_tcl.txt};

\addplot [thick, loosely dashed ,draw=red] table[x index=0,y index=2]{R_tcl.txt};

\addplot [thick, loosely dashed  ,draw=black] table[x index=0,y index=3]{R_tcl.txt};

\end{axis}
\end{tikzpicture}
\caption{Total power consumption $U$ and allocated response $R$ (MW) of the TCLs after 75 algorithm iterations.}
\label{conso comp}
\end{minipage}\hspace*{0.1cm}
\begin{minipage}[t]{.5\linewidth}

\begin{tikzpicture}[font=\scriptsize]
\begin{axis}[ymax = 120,
legend entries={$p_{\sigma = 0}$,$p_{\sigma = 1}$,$p_{\sigma = 2}$, $\rho_{\sigma = 0}$,$\rho_{\sigma = 1}$,$\rho_{\sigma = 2}$},legend style={at={(1,1)},anchor=north east},legend columns=3,
grid=major,height=7cm,width=7cm,
xlabel={Time (h)}
]

\addplot [thick, draw=blue] table[x index=0,y index=1]{p.txt};

\addplot [thick, draw=red] table[x index=0,y index=2]{p.txt};

\addplot [thick, draw=black] table[x index=0,y index=3]{p.txt};

\addplot [thick,loosely dashed,draw=blue] table[x index=0,y index=1]{r.txt};

\addplot [thick ,loosely dashed,draw=red] table[x index=0,y index=2]{r.txt};

\addplot [thick, loosely dashed,draw=black] table[x index=0,y index=3]{r.txt};

\end{axis}

\end{tikzpicture}
\caption{Electricity price $p$ and response availability price $\rho$ (\textsterling /MWh) after 75 algorithm iterations.}
\label{price comp}
\end{minipage}
\end{figure}
\end{center}

\begin{center}
\begin{figure}
\begin{minipage}[t]{.45\linewidth}

\begin{tikzpicture}[font=\scriptsize]

\begin{axis}[ymax = 2150,
legend entries={$CCGT_{\sigma = 0}$,$CCGT_{\sigma = 1}$,$CCGT_{\sigma = 2}$,
$Wind_{\sigma = 0}$,$Wind_{\sigma = 1}$,
$Wind_{\sigma = 2}$}
,legend style={at={(1,1)},anchor=north east},legend columns=2,
grid=major,height=7cm,width=7cm,
xlabel={Time (h)},
]

\addplot [thick,draw=blue] table[x index=0,y index=1]{G_reduced.txt};

\addplot [thick,draw=red] table[x index=0,y index=2]{G_reduced.txt};

\addplot [thick,draw=black] table[x index=0,y index=3]{G_reduced.txt};

\addplot [ thick, loosely dashed , draw=blue] table[x index=0,y index=4]{G_reduced.txt};

\addplot [thick, loosely dashed, draw=red] table[x index=0,y index=5]{G_reduced.txt};

\addplot [thick, loosely dashed, draw=black] table[x index=0,y index=6]{G_reduced.txt};

\end{axis}
\end{tikzpicture}
\caption{
Deviation of generation profiles (MW) from the ``business-as-usual'' scenario during the first hours of the day, considering three different values of temperature volatility $\sigma$.}
\label{prod comp}

\end{minipage}\hspace*{0.2cm}
\begin{minipage}[t]{.5\linewidth}

\begin{tikzpicture}[font=\scriptsize]

\begin{axis}
[ymax = 200,
legend entries={
 $CCGT_{\sigma = 0}$,
$CCGT_{\sigma = 1}$,$CCGT_{\sigma = 2}$,$CCGT_{no flex}$
}
,legend style={at={(1,1)},anchor=north east}, legend columns=2,
grid=major,height=7cm,width=7cm,
xlabel={Time (h)},
]
\addplot [thick, draw=blue] table[x index=0,y index=1]{FR_reduced.txt};

\addplot [draw=red] table[x index=0,y index=2]{FR_reduced.txt};

\addplot [draw=black] table[x index=0,y index=3]{FR_reduced.txt};

\end{axis}
\end{tikzpicture}
\caption{
Deviation of Frequency Response (MW) allocated by CCGT technology with respect to the ``business-as-usual'' scenario during the first hours of the day,  considering three different values of temperature volatility $\sigma$.}
\label{fr comp}
\end{minipage}
\end{figure}
\end{center}

A comparison of the system costs (i.e. UC solution) between the ``flexibility scenario'' (FS) and the ``Business-as-usual'' (BAU) framework
is provided in Table \ref{syst cost}.
As expected, costs are lower in the FS, as the flexibility of the TCLs positively supports system operation, allowing to replace gas generation from OCGT and CCGT plants with cheaper wind energy. The reduction is higher (about 1.9\%) for $\sigma = 0$ with respect to the cases with $\sigma = 1$ or $\sigma =2$ (about 1.6\% and 1.2\%, respectively). This confirms previous indications that TCLs tend to be more flexible when the volatility of their internal temperature is lower.
\begin{table}

 \begin{center}
\begin{tabular}{|c|c|c|c|}
  \hline
  & $\sigma = 0$ & $\sigma = 1$ & $\sigma = 2$ \\
    \hline  
BAU & $2.770\times 10^7$ &$2.770\times 10^7$ & $2.772\times 10^7$\\
  \hline    
FS & $2.719\times 10^7$ & $2.725 \times 10^7$& $2.740\times 10^7$\\
  \hline
\end{tabular}
\end{center}
 \caption {Minimized system costs in (\textsterling )} \label{syst cost} 
\end{table}

\section{Conclusions}
Randomness and high dimensionality usually make the resolution of an optimization problem quite difficult. However, in the specific case of convex aggregative control problems, we have shown that, under independent noise assumptions, one can take advantage of the high dimension to approximate accurately the original Problem $(P_1)$ by a decentralized Problem $(\hat{P}_2)$, whose numerical resolution is more tractable. We highlight the fact that the approximation error is of order $\frac{1}{n}$, where $n$ is the number of agents. The extension of this approach to stochastic control problems with common noise or to non convex problems may be challenging but interesting topics for further work.

\appendix
\section{Appendix}\label{append}

\begin{lemma}\label{quad_growth_lemma}
Let $H$ be a Hilbert space and $f:H\to \mathbb{R}$ be l.s.c. and convex. The function $f$ has at most quadratic growth if and only if its subgradient has linear growth.
\end{lemma}

\begin{proof}
  Let the subgradient have linear growth, that is,
  $\Vert q\Vert_H \leq c_1(1+\Vert x \Vert_H)$
  whenever $x \in H$ and $q\in \partial f(x)$.
  Then
  $f(x) \leq f(0) + \langle q, x \rangle_H
  \leq f(0) + \Vert q\Vert_H \Vert x\Vert_H
  \leq c_2 (1+\Vert x \Vert_H^2)$, so that $f$
  has at most quadratic growth.

Conversely, let $f$ have at most quadratic growth. Since $f$ is convex, one has for all $x\in H$ and $q_0\in \partial f(0)$:
$$
f(x)\geq f(0) + \langle q_0,x \rangle_H \geq  -c_3(1+ \Vert x \Vert_H^2),
$$
where $c_3>0$ depends only on $f(0)$ and $q_0$. Then, using the growth assumption on $f$ and the inequality above, one gets for all $x,y\in H $ and $q\in\partial f(x)$:
$$
c_4(1+ \Vert y \Vert_H ^2) \geq f(y) \geq f(x) + \langle q, y-x \rangle_H
\geq -c_3(1+\Vert x \Vert_H^2) + \langle q, y-x \rangle_H.
$$
Take $y = x + \alpha q$, with $\alpha\in (0,1)$, we get
$$
2c_4(1 + \Vert x \Vert_H^2 + \alpha^2 \Vert q \Vert_H^2)\geq 
-c_3(1+\Vert x \Vert_H^2) + \alpha\Vert q\Vert_H^2
$$
so that
$( \alpha - 2c_4\alpha^2) \Vert q\Vert_H^2
\leq (2c_4+c_3)(1+\Vert x \Vert_H^2)$. 
Take $\alpha = 1/(4c_4)$, then
$\alpha - 2c_4 \alpha^2 =1 /(8c_4) >0$ and then
$$
 \Vert q \Vert_H^2\leq 
 8c_4(2c_4+c_3)(1+ \alpha\Vert x\Vert_H^2)
$$
and the conclusion follows.
\end{proof}

 \end{document}
\today}